\documentclass[11pt,letterpaper]{amsart}

\usepackage[margin=1.2in]{geometry}                

\usepackage{mathabx}
\usepackage{amsmath,amssymb,amsthm,amscd}
\usepackage{colonequals}

\usepackage{graphicx}
\usepackage{tikz}
\usepackage{tikz-cd}
\usetikzlibrary{decorations.markings}
\usetikzlibrary{plotmarks}
\usepackage{epstopdf}
\usepackage{import} 

\usepackage{multicol}

\usepackage{url}
\usepackage{color}
\usepackage[colorlinks,citecolor=blue,linkcolor=red!80!black]{hyperref}
\usepackage[nameinlink]{cleveref}
\usepackage{comment}

\usepackage{marginnote}
\usepackage{todonotes}

\theoremstyle{plain}
\newtheorem{theorem}{Theorem}[section]
\newtheorem{lemma}[theorem]{Lemma}
\newtheorem{corollary}[theorem]{Corollary}

\newtheorem{proposition}[theorem]{Proposition}

\newtheorem{fact}[theorem]{Fact}

\crefname{claim}{Claim}{Claims}
\newtheorem*{claim*}{Claim}
\newtheorem*{introdet*}{{\Cref{thm:det_form}}}
\newtheorem*{intro-orientable*}{{\Cref{thm:orientable}}}

\newcounter{theoremalph}

\newtheorem{thmAlph}[theoremalph]{Theorem}

\theoremstyle{definition}
\newtheorem{definition}[theorem]{Definition}

\newtheorem{remark}[theorem]{Remark}

\newtheorem{example}[theorem]{Example}

\makeatletter
\let\c@equation\c@theorem
\makeatother
\numberwithin{equation}{section}

\newcommand{\wh}{\widehat}
\newcommand{\ol}{\overline}

\newcommand{\wt}{\widetilde}
\newcommand{\mf}{\mathfrak}

\newcommand{\mc}{\mathcal}
\newcommand{\BNS}{\mathrm{BNS}}
\newcommand{\define}{\emph}

\newcommand{\Z}{\mathbb{Z}}

\newcommand{\R}{\mathbb{R}}
\newcommand{\abs}[1]{\left\vert#1\right\vert}
\newcommand{\norm}[1]{\left\Vert#1\right\Vert}
\newcommand{\inv}{^{-1}}

\newcommand{\supp}{\mathrm{supp}}
\newcommand{\flow}{\psi}
\newcommand{\dt}{z} 
\newcommand{\posgen}{\bar{\dt}} 
\newcommand{\free}{\mathbb{F}} 
\newcommand{\map}{g} 

\title{Orientable maps and polynomial invariants of free-by-cyclic groups}
\author[Dowdall]{Spencer Dowdall}
\address{Department of Mathematics, Vanderbilt University, Nashville, TN}
\email{spencer.dowdall@vanderbilt.edu}
\author[Gupta]{Radhika Gupta}
\address{School of Mathematics, Tata Institute of Fundamental Research, Mumbai}
\email{rgupta@tifr.res.in}
\author[Taylor]{Samuel J. Taylor}
\address{Department of Mathematics, Temple University, Philadelphia, PA}
\email{samuel.taylor@temple.edu}
\date{}                                           

\begin{document}

\begin{abstract}
We relate the McMullen polynomial of a free-by-cyclic group to its Alexander polynomial. To do so, we introduce the notion of an \emph{orientable} fully irreducible outer automorphism $\varphi$ and use it to characterize when the homological stretch factor of $\varphi$ is equal to its geometric stretch factor.
\end{abstract}

\maketitle

\setcounter{tocdepth}{1}
\tableofcontents

\section{Introduction}
This paper introduces the property of \emph{orientability} for fully irreducible free group automorphisms and characterizes how it is reflected in both the stretch factors of the automorphism and in the Alexander and McMullen polynomials of the associated free-by-cyclic group.

To motivate our work, let us recall the theory from the mapping class group setting.
A pseudo-Anosov homeomorphism $\varphi$ of an orientable 
hyperbolic surface comes equipped with two notions of stretch factor:
The \emph{geometric} stretch factor or dilatation $\lambda_\varphi$ is the value by which $\varphi$ stretches the leaves of its expanding
foliation; $\lambda_\varphi$ is also characterized as the exponential growth rate of the geodesic length of any curve under iteration by $\varphi$, and its logarithm $\log\lambda_\varphi$ is both the topological entropy of $\varphi$ and the translation length for its action on Teichm\"uller space.
The \emph{homological} stretch factor $\rho_\varphi$ is instead the spectral radius for the action of $\varphi$ on the first homology of the surface.
It is well-known that these two numbers are equal if and only if the 
invariant foliations of $\varphi$ are transversely orientable. Indeed, that transverse orientability implies equality of the stretch factors was observed by Thurston  \cite{thurston1988geometry}, and the reverse implication was proven by Band and Boyland  \cite[Lemma 4.3]{BandBoyland}. 

In fact, a much stronger symmetry persists when $\varphi$ preserves these transverse orientations.
Associated to the mapping torus of a pseudo-Anosov homeomorphism there are two important polynomial invariants that arise: the classical Alexander polynomial \cite{alexander1928topological}, which captures homological information about monodromies, and the Teichm\"uller polynomial \cite{McMullen--poly_invariants} of the fibered cone, which encodes the associated geometric stretch factors.
In the orientable case, McMullen proved these invariants are closely related in that the Alexander polynomial divides the Teichm\"uller polynomial \cite[Theorem 7.1]{McMullen--poly_invariants}. 
More recently, Parlak \cite{Parlak} extended McMullen's work to explicitly determine the quotient polynomial, giving a precise relation between the two polynomials and the  associated stretch factors.

\subsection{Orientable automorphisms}
In this article, we extend the above results to the free group setting, replacing the pseudo-Anosov homeomorphism with a fully irreducible automorphism 
of a finite rank free group $\free$ and replacing the mapping torus $M_\varphi$ by the associated free-by-cyclic group.
In fact, we define what it means for a fully irreducible endomorphism to 
be \emph{orientable}
and prove this occurs exactly when the geometric and homological stretch factors agree.

The condition is best understood in the context of graph maps. 
For a graph map $f\colon G\to G$,
we define the geometric stretch factor $\lambda_f$ to be the spectral radius of the transition matrix of $f$, and the homological stretch factor $\rho_f$ to be the spectral radius of the induced map $f_*$ on $H_1(G;\R)$. 
An \emph{orientation} on $G$ is a choice of positive orientation for each edge of $G$. We say $f$ is \define{positively orientable} (resp.\ \define{negatively orientable}) if there exists an orientation on $G$ such that every positive edge maps to a positive (resp.\ negative) edge path. 
We often abbreviate this terminology to pos/neg-orientable or $\pm1$--orientable. 
Combining these notions, the map is called \define{orientable} if it respects some orientation on $G$, in that it is either positively orientable or negatively orientable.
See \S\ref{sec:orientable_maps} for additional details. 

\begin{thmAlph}\label{thm:orientable}\label{th:stretch_or}
Suppose $f \colon G \to G$ is a graph map whose transition matrix has a positive power. 
Then $\lambda_f = \rho_f$ if and only if $f$ is orientable. 
Moreover, $\lambda_f$ (resp.\ $-\lambda_f$) is an eigenvalue of $f_*$ if and only if $f$ is pos-orientable (resp.\ neg-orientable).
\end{thmAlph}

By combining this with well-known properties of fully irreducible endomorphisms and their expanding laminations (see \S\ref{sec:endo}), we obtain the following corollary, which is a simplified statement of \Cref{th:stretch_or_auto}:
\begin{thmAlph}\label{thm:orientable-auto}
For a fully irreducible free group endomorphism $\varphi$,
the following are equivalent:
\begin{itemize}
\item the homological and geometric stretch factors of $\varphi$ are equal;
\item some (every) irreducible 
train track representative of $\varphi$ is orientable;
\item the expanding lamination $\mc L^+_\varphi$ of $\varphi$ is orientable.
\end{itemize}
\end{thmAlph}

\subsection{Splittings and the BNS-invariant}
\label{sec:intro_BNS}
To frame our results relating Alexander and McMullen polynomials, we start with a corollary characterizing when geometric and homological stretch factors agree for monodromies of free-by-cyclic groups.

Any group endomorphism $\phi\colon B\to B$ determines a \emph{generalized HNN extension}
\[B\ast_\phi \colonequals  \langle B, t \mid t\inv bt = \phi(b)\text{ for all }b\in B\rangle\]
that comes equipped with a projection $B\ast_\phi\to \Z$ sending $B$ to $0$ and $t$ to $1$; we call this element of $\mathrm{Hom}(B \ast_\phi,\Z) = H^1(B \ast_\phi;\Z)$ the \define{dual cohomology class}.

In the case that $\Gamma$ is finitely generated, there is an open, $\R_+$--invariant subset $\BNS(\Gamma)$ of $H^1(\Gamma;\R)\setminus\{0\}$ that governs finite generation of kernels of such maps $\Gamma\to \Z$. 
This \define{Bieri--Neumann--Strebel invariant} \cite{BNS} is defined to contain precisely those rational classes $u\colon \Gamma\to \mathbb{Q}$ for which $\ker(u)$ is finitely generated over the $u$--positive monoid, 
meaning there is an element $t\in \Gamma$ with $u(t) > 0$ so that $\ker(u)$ is generated by the positive conjugates $\{t^n X t^{-n} \mid n\ge 0\}$ of some finite subset $X\subset\ker(u)$.
Unpacking this, one finds that
$u\colon \Gamma\to \Z$ lies in $\BNS(\Gamma)$ if and only if it is the dual class of an HNN splitting $\Gamma\cong B\ast_\phi$ over a \emph{finitely generated} base group $B$ \cite[Proposition 4.3]{BNS}.
We will say that any such splitting $\Gamma\cong B\ast_\phi$ (with $B$ finitely generated) is \emph{dual} to $u$
and call $\phi$ a \emph{monodromy} associated to $u$.

Now suppose that $\Gamma$ is free-by-cyclic. For any splitting $\Gamma \cong B*_\phi$ dual to a class $u \in \BNS(\Gamma)$, the group $B$ 
is free.  
Although the monodromy $\phi$ is not uniquely determined by $u$, 
its geometric and homological stretch factors are determined by $u$ and 
these are denoted by $\lambda(u)$ and $\rho(u)$, respectively. See \S\ref{sec:relating} for details.

\begin{thmAlph}
\label{thm:intro-identify_orientable_classes}
Consider the free-by-cyclic group $\Gamma = \free\ast_\varphi$, where $\varphi$ is a fully irreducible automorphism, and let $\mathcal{C}$ be the component of $\BNS(\Gamma)$ containing the dual class of $\free\ast_\varphi$.
 Then the set of all primitive integral  $u \in \mc C$ such that $\lambda(u) = \rho(u)$ is either $(1)$ the entire cone, $(2)$ exactly one residue class mod $2$, or $(3)$ empty.
\end{thmAlph}

When $\varphi$ is additionally atoroidal, it is known that \emph{every} monodromy associated to any primitive integral class $u\in \mc C$ is also fully irreducible \cite{DKL-endos,mutanguha--invariance_of_iwip}; hence in light of \Cref{thm:orientable}, the classes with $\lambda(u) = \rho(u)$ described in \Cref{thm:intro-identify_orientable_classes} are precisely those whose monodromy endomorphisms are orientable; see \Cref{rem:iso_stretch_implies_orient_endo}.

\subsection{Relating the Alexander and McMullen polynomials}
\label{sec:relating_alex_and_mcmullen_polys}
The trichotomy in \Cref{thm:intro-identify_orientable_classes} is obtained by explicitly relating two polynomial invariants attached to $\Gamma = \free\ast_\varphi$.
The first of these, the Alexander polynomial, was originally defined by Alexander in 1928 \cite{alexander1928topological} as a knot invariant, and its definition 
was later extended by McMullen to all finitely generated groups \cite{McMullen--Alexander_polynomial}. 
It is an invariant of the group $\Gamma$ and is formally defined, up to a unit, as an element $\Delta_\Gamma$ of the ring $\Z[H]$, where $H=H_1(\Gamma;\Z)/\mathrm{torsion}$; see \S\ref{sec:alex_poly_of_free-by-cyclic} or the next subsection \S\ref{sec:intro-computing_alex} below for details.

The second of these, the McMullen polynomial, was defined by Dowdall--Kapovich--Leininger in \cite{DKL2} as an analog of McMullen's Teichm\"uller polynomial from \cite{McMullen--poly_invariants}. 
It is defined in terms of an irreducible train track representative $f\colon G\to G$ of a fully irreducible automorphism $\varphi$ of $\pi_1(G)\cong \free$ and an associated \define{folded mapping torus} $X^\dagger$ that comes equipped with a natural semiflow $\flow^\dagger$ and is homotopy equivalent to the usual mapping torus of $f$.
The \define{McMullen polynomial} $\mf m\in \Z[H]$ is an algebraic invariant of this dynamical system $(X^\dagger,\flow^\dagger)$. The key features of $\mf m$ are that it picks out the component $\mc C$ of the BNS-invariant containing the dual class of the splitting $\Gamma = \free\ast_{\varphi}$ and calculates the geometric stretch factors $\lambda(u)$ of the classes $u\in \mc C$, as in \Cref{thm:intro-identify_orientable_classes}. We refer the reader to \cite{DKL2} and \S\ref{sec:fmt} for more details.

To connect these, in \S\ref{sec:vertex_poly} we introduce an additional \emph{vertex polynomial}
\[\mf p = \prod_i (1-z_i)\in \Z[H],\]
of the dynamical system $(X^\dagger, \flow^\dagger)$, where $\{z_i\}$ are the elements in $H = H_1(X^\dagger;\Z)/\text{torsion}$ represented by the finitely many closed orbits of $\flow^\dagger$ through the vertices of $X^\dagger$.
Our next theorem shows that these three polynomials are related in a precise way that moreover reflects the positive, negative, or non-orientability of $f$:

\begin{thmAlph}\label{th:relation1}
Let $f \colon G \to G$ be an irreducible train track map 
representing a fully irreducible automorphism $\varphi$ of $\pi_1(G)\cong \free$,
with associated free-by-cyclic group $\Gamma$ and folded mapping torus $(X^\dagger, \flow^\dagger)$. 
 If  $\mathrm{rank}(H_1(X, \R)) \geq 2$, then up to a unit:
 \begin{enumerate}
 \item if $f$ is pos-orientable, then $\mf m = \Delta_{\Gamma} \cdot \mf p$, and
 \item if $f$ is neg-orientable, then there is an \emph{orientation involution} $\iota \colon \Z[H]\to \Z[H]$ for which  $ \iota (\mf m) = \Delta_\Gamma \cdot \mf p$.
 \item in general, $\mf m = \Delta_{\Gamma} \cdot \mf p$ mod 2.
 \end{enumerate}
 If instead $\mathrm{rank}(H_1(\Gamma;\R)) = 1$, then each equations holds after multiplying $\mf m$ by $(\dt - 1)$, where $\dt$ generates $H_1(\Gamma, \Z)$.
\end{thmAlph}

Moreover, \Cref{th:relation1} applies not only to the original monodromy, but also to any monodromy dual to a primitive integral class in the corresponding component $\mc C$ of the BNS-invariant, so that its 3 conclusions pair with the trichotomy of \Cref{thm:intro-identify_orientable_classes}. Indeed, 
the involution $\iota$ in the neg-orientable case comes from an \emph{orientation class} that depends only on $\mc C$ (\Cref{lem:or_class}).
This more general perspective is taken in 
\Cref{th:polynomials,th:anti_cone,th:polynomials_mod2}, which, respectively, give more detailed treatments of the conclusions (1), (2) and (3) and, collectively, prove \Cref{th:relation1}.

\subsection{Computing the Alexander polynomial}
\label{sec:intro-computing_alex}
Although the Alexander polynomial can always be efficiently computed from a group presentation via Fox Calculus (see \cite{Button} or \cite{hironaka2011computing} for relevant examples), our proof of \Cref{th:relation1} 
requires a computation of $\Delta_\Gamma$ that is compatible with the definition of $\mf m$.
Our final result provides this by giving an explicit `determinant formula' for the Alexander polynomial that is inspired by McMullen's formula for the Teichm\"uller polynomial.

To explain, recall that any graph map $f\colon G\to G$ determines a \emph{mapping torus}
\[
X = X_f \colonequals \frac{G \times [0,1]} { (x,1) \sim (f(x),0)},
\] 
whose fundamental group $\Gamma$ splits as a generalized HNN extension, which we write as $\Gamma \cong \pi_1(G)\ast_f$. Here, the associated dual class of the splitting is obtained by pulling back the fundamental class of the circle under the associated fibration $X \to \mathbb{S}^1$. 

Our determinant formula computes the Alexander polynomial $\Delta_\Gamma$ of $\Gamma = \pi_1(X)$ directly in terms of the graph map $f$. 
Let $\wt{X}\to X$ be the universal free abelian cover of $X$ with deck group $H$. Fix a connected component $\wt{G}_0$ of the full preimage of $G$ in $\wt{X}$ and fix a lift $\wt{f} \colon \wt{G}_0 \to \wt{G}_0$ of $f$. 
The deck group of $\wt{G}_0 \to G$ is denoted by $K$ and is naturally identified with the image of the inclusion induced homomorphism $\pi_1(G) \to \pi_1(X) \to H$.
Moreover, the choice of lift $\wt f \colon \wt G_0 \to \wt G_0$ determines (see \S\ref{sec:abelian_cover}):
\begin{itemize}
\item a splitting $H = K \oplus \langle \dt \rangle$, where $\dt$ translates against the lifted semiflow,
\item $\Z[K]$--valued matrices $\wt M$ and $\wt P$ representing the action of $\wt f$ on the free $\Z[K]$--modules
 of $1$--chains and $0$--chains of $\wt G_0$, respectively. 
\end{itemize}

\begin{thmAlph}[Determinant formula for $\Delta_{\Gamma}$] \label{thm:det_form}\label{th:det_form}
Let $f \colon G \to G$ be a graph map whose mapping torus $X$ has fundamental group $\Gamma$. 
Then up to a unit in $\mathbb{Z}[H]$, 
\[
\Delta_\Gamma = \frac{\det(\dt I -\wt M)}{\det(\dt I-\wt P)} \cdot r,
\]
where $r=1\in \Z[H]$ if $\mathrm{rank}(H_1(X, \R)) \ge 2$ and $r = (\dt-1)$ 
if $\mathrm{rank}(H_1(X, \R)) =1$.
\end{thmAlph}

\subsection*{Acknowledgements.}
We thank Anna Parlak for helpful comments that lead to \Cref{th:polynomials_mod2} and the referee for useful suggestions. SD was partially supported by NSF grants DMS-1711089 and DMS-2005368. RG and ST were partially supported by the Sloan Foundation, and ST was partially supported by NSF grant DMS-2102018.

\section{Orientable graph maps and stretch factors}\label{sec:orientable_maps}
Let $G$ be a finite graph with vertex set $V = V(G)$ and edge set $E = E(G)$.
A \define{graph map} $f \colon G \to G$ is a continuous map sending vertices to vertices and edges to nondegenerate edge paths.
 For careful definitions of standard notions like `graph' or `edge path,' we refer the reader to \cite[\S2]{DKL2}.
 
An \emph{orientation} on $G$ is a choice of a positive orientation for each edge of $G$. With such a choice, we say an edge path $\gamma$ is \emph{positive} if it crosses each edge with its positive orientation, and is \emph{negative} if it crosses each edge with its negative orientation. In general, if $\gamma\colon [0,1]\to G$ is an edge path, then $\overline{\gamma}$ will denote the same path with the reverse orientation.

\begin{definition}[Orientability]
\label{def:orientable}
A graph map is \define{positively orientable} (resp.\ {negatively orientable}) if there exists an orientation on $G$ such that every positive edge is mapped to a positive (resp.\ negative) edge path. 
We abbreviate these as pos/neg or $\pm1$--orientable.
The map is \define{orientable} if it is either pos-orientable or neg-orientable.
\end{definition} 

Note that if $f$ is neg-orientable, then $f^2$ is pos-orientable.
Since a consistently oriented edge path in $G$ cannot backtrack, we observe that an  orientable map is necessarily a train track map (see \S\ref{sec:tt_reps}).
See \S\ref{sec:examples} for examples.

Choosing an orientation on each edge gives a basis for the vector space $C_1(G; \R)$ of simplicial $1$--chains. Under this identification $C_1(G; \R)\cong  \R^E$, the induced map on simplicial $1$--chains is expressed as a matrix $M\colon  \R^E\to  \R^E$ with integer coefficients.
The real first homology of $G$ is the subspace $H_1(G; \R)\le  \R^E$ of $1$--cycles, and we denote the induced map by $f_* \colon H_1(G;  \R) \to H_1(G;  \R)$. 
The \emph{homological stretch factor} of $f$ is defined to be the spectral radius $\rho_f$ of $f_*$.
Note that this also equals the spectral radius of $M$; indeed
the eigenvalues of $M$ that are not eigenvalues of $f_*$ 
are eigenvalues of the action of $f$ on $C_0(G; \R)$, which are roots of unity.

For an arbitrary graph map $f \colon G \to G$, we will denote the transition matrix of $f$ by $A$, whose $(i,j)$ entry is the number of times the edge path $f(e_j)$ crosses the edge $e_i$ with either orientation. The \emph{(geometric) stretch factor}
$\lambda_f$ of $f$ is defined to be the spectral radius of $A$. Equivalently, since $A$ is non-negative, $\lambda_f$ is the eigenvalue of $A$ of largest modulus. 
Since the absolute values of the entries of $M$ are bounded by the entries of $A$,  we have $\rho_f \leq \lambda_f$ by Gelfand's formula for the spectral radius: for any square matrix $A$, the spectral radius of $A$ is equal to $\lim_{r \to \infty} \|A^r\|^{1/r}$. The following fact relating $M$ and $A$ follows from the definitions. 

\begin{fact} \label{fact:matrix}
If $f \colon G \to G$ is $\pm1$--orientable, then $M= \pm A$ after a choice of orientation.
\end{fact}

Recall that a non-negative square matrix $A$ is \define{irreducible} if for each pair of indices $(i,j)$ there exists $k\ge 1$ such that the $(i,j)$ entry of $A^k$ is nonzero. The matrix is moreover \define{primitive} if there exists $k \ge 1$ such that $A^k$ is a positive matrix. Note that $A$ is primitive if and only if $A^k$ is irreducible for all $k\ge 1$. We extend this terminology by saying a graph map $f\colon G\to G$ is
\define{irreducible}/\define{primitive} if its transition matrix is irreducible/primitive. Relatedly, $f$ is  \define{expanding} if for each edge $e$ of $G$ the combinatorial lengths of the edge paths $f^k(e)$ tend to infinity as $k\to \infty$. By the Perron-Frobenius theorem (see \cite{Gantmacher-TheoryMatrices1and2}), if $f$ is a primitive graph map, then $\lambda_f>1$. 

The main goal of this section is to prove the following theorem from the introduction:
\begin{intro-orientable*}
Suppose that $f \colon G \to G$ is a primitive graph map. 
Then $\lambda_f = \rho_f$ if and only if $f$ is orientable. 
Moreover, $\lambda_f$ (resp.\ $-\lambda_f$) is an eigenvalue of $f_*$ if and only if $f$ is pos-orientable (resp.\ neg-orientable).
\end{intro-orientable*}

\subsection{The oriented edge-double}
Motivated by the use of double branched covers in \cite[Lemma 4.3]{BandBoyland},
we define the \emph{oriented edge double} $\wh G$ of $G$ as follows:  Let $p \colon \wh G \to G$ be a bijection on vertices and $2$-to-$1$ over edges such that
for each oriented edge $e$ of $G$, there are oriented edges $e_+$ and $e_-$ of $\wh G$ labeled so that $p(e_+) = e$ and $p(e_-) = \ol e$.
There is a natural continuous involution $\sigma\colon \wh G\to \wh G$ defined by $e_+\mapsto \ol e_-$ and $e_-\mapsto \ol e_+$. Notice that by construction $\sigma$ is neg-orientable and satisfies $p = p\circ \sigma$.

The key feature of $\wh G$ is that \emph{every} edge path in $G$ has a unique \emph{positive lift} to $\wh G$. That is, for any edge path $\gamma\colon [0,1]\to G$ there is a unique positive edge path $\wh \gamma\colon [0,1]\to \wh G$ such that $\gamma = p\circ \wh\gamma$. Similarly $\gamma$ has a unique \emph{negative lift}, which is simply the reverse of the unique positive lift of $\ol \gamma$ and is given by $\sigma\circ \wh\gamma$.
In particular, $e_+$ and $e_-$ are, respectively, the positive lifts of $e$ and $\ol e$.

Since $\sigma$ is an involution, the eigenvalues of 
$\sigma_*\colon H_1(\wh G ; \R)\to H_1(\wh G ; \R)$  are $\pm 1$ and we write $E_{1}$, $E_{-1}$ for the corresponding eigenspaces.

\begin{lemma}\label{lem:splitting}
We have $H_1(\wh G;  \R)= E_1\oplus E_{-1}$. Furthermore, $p_*\colon H_1(\wh G;  \R)\to H_1(G;  \R)$ restricts to an isomorphism $E_1\stackrel{\cong}{\to} H_1(G)$ and has kernel $\ker(p_*) = E_{-1}\cong C_1(G;\R)$.
\end{lemma}
\begin{proof}
For any class $c\in H_1(\wh G)$, we clearly have $c + \sigma_*(c)\in E_1$ and $c - \sigma_*(c)\in E_{-1}$. Since
\[c = \frac{c + \sigma_*(c)}{2} + \frac{c - \sigma_*(c)}{2},\]
the direct sum splitting $H_1(\wh G) = E_1\oplus E_{-1}$ follows. The containment $E_{-1}\le \ker(p_*)$ follows immediately from $p = p\circ \sigma$, since if $c\in E_{-1}$ then $p_*(c) = p_*(\sigma_*(c)) = p_*(-c) = -p_*(c)$. 

There are natural injective maps $s_\pm\colon C_1(G ; \R)\to C_1(\wh G ; \R)$ on $1$--chains defined on each basis edge $e\in C_1(G;  \R)$ by $s_\pm(e) = \frac{1}{2}(e_+ \mp e_-) = \frac{1}{2}(e_+ \pm \sigma_*(e_+))$. The image of $s_+$ is contained  in the fixed set of $\sigma_*$, that is $\sigma_*\circ s_+ = s_+$. 
Since $s$ sends cycles to cycles, it therefore restricts to a map $s_+\colon H_1(G; \R)\to E_1 \le H_1(\wh G ;  \R)$.
Further, we see that for any $c\in C_1(G ;  \R)$ we have $p_*(s_+(c))= c$.
Similarly the image of $s_-$ is clearly contained in $E_{-1}\le H_1(\wh G)$.
Now  any $\wh c_\pm \in E_{\pm1}$ necessarily has the form $\wh c_\pm = \sum_{e \in G} a_e (e_+ \mp e_-) = \sum_{e\in G} 2 a_e s_\pm(e)$. In particular $s_-$ is surjective onto $E_{-1}$, giving an isomorphism $C_1(G;\R)\to E_{-1}$. 
We also see that $p_* (\wh c_+ ) = \sum_{e \in G} 2a_e e$ and hence $s_+(p_*(\wh c_+)) = \sum_{e \in G} 2a_e s_+(e) = \wh c_+$.
It follows that $p_*$ restricts to an isomorphism on $E_1\to H_1(G;\R)$ and that additionally $\ker(p_*) \le E_{-1}$.
\end{proof}

Any graph map $f\colon G\to G$ has a unique pos-orientable lift $\wh f\colon \wh G \to \wh G$, preserving the specified orientation of $\wh G$, satisfying $p\circ \wh f = f \circ p$. This sends $e_+$ to the unique positive lift of $f(e)$, and sends $e_-$ to the unique positive lift of $f(\ol e)$. More generally, for any positive edge path $\gamma$ in $\wh G$ we have 
\[\wh f\circ \gamma = \widehat{f \circ p\circ \gamma}.\]
We denote the transition matrix of $\wh f$ by $\wh A$ and note that $\wh A$ is also the induced map on simplicial $1$-chains (see \Cref{fact:matrix}).
Notice that for any positive edge $e$ of $\wh G$, the paths $\sigma(\wh f(e))$ and $\wh f(\sigma(e))$ are both negative lifts of $f(p(e))$. As such lifts are unique, we conclude that $\wh f$ and $\sigma$ commute:
\[\wh f \circ \sigma = \sigma \circ \wh f.\]
Consequently, $\wh f_*$ necessarily preserves the splitting $H_1(\wh G) = E_{1}\oplus E_{-1}$ from \Cref{lem:splitting}; in fact $\wh f_* = f_*\oplus A$ with the transition matrix $A$ describing the action on $C_1(G;\R)\cong E_{-1}$.

\begin{lemma} \label{lem:prim_lift}
Let $f \colon G \to G$ be a primitive graph map and let $\wh f\colon \wh G \to \wh G$ be its orientation lift. Then $\lambda_{\wh f} =\lambda_f$ and $\wh f$ is primitive if and only if $f$ is not orientable.
\end{lemma}

\begin{proof}
As noted previously, for any square matrix $B$ with spectral radius $\mu$, we have that  $\mu = \lim_{r \to \infty} \Vert B^r \Vert^{1/r}$ where $\Vert \cdot \Vert$ is any matrix norm. 
Let us take $\Vert \cdot \Vert$ to be the maximum absolute column sum (i.e.~the $\ell^1$ operator norm).
Now observe that the corresponding column sums of (all powers of) $\wh A$ and $A$ are equal. Indeed, this follows from the definitions of $\wh A, A$ and the fact that for each edge $e$ of $G$, $p \circ \wh f^n(e_+) = f^n(e)$ and $p \circ \wh f^n(e_-) = f^n(\ol e)$. Therefore $\Vert{\wh A^r}\Vert = \Vert{A^r}\Vert$ for all $r\ge 1$ and we conclude $\lambda_{\wh f} =\lambda_f$.

Let $D(A)$ and  $D(\wh A)$ be the directed graphs such that their adjacency matrices are $A$ and $\wh A$, respectively. The vertices of $D(A)$ and $D(\wh A)$ are labeled by the edges of $G$ and $\wh G$, respectively. 
There is a natural simplicial map $D(\wh A) \to D(A)$ that is $2$-to-$1$ over vertices (induced by $\wh G \to G$) and also $2$-to-$1$ over edges. It is an isomorphism on the link of each vertex, so it is a (possibly disconnected)  $2$-fold covering. 
It follows that $f$ is orientable if and only if $D(\wh A)$ is disconnected, and each component of $D(\wh A)$ gives an $f$-invariant orientation of $G$.

Since $A$ is irreducible, $D(A)$ is strongly connected, 
 i.e.~for any two vertices, there is a directed path from one to the other.
We now understand the connectivity of $D(\wh A)$. By path lifting, any directed path from $e$ to $e'$ in $D(A)$, lifts to a directed path in $D(\wh A)$ from $e_+$ to either $e'_+$ or $e'_-$ (and similarly for $e_-$). Moreover, if there is a directed path from $e_+$ to $e'_\pm$, then there is a corresponding directed path from $e_-$ to $e'_\mp$. This can be obtained by applying the deck transformation of $D(\wh A) \to D(A)$.
Hence, there are directed loops starting at each vertex of $D(\wh A)$, and if there is a directed path from $v$ to $w$ then there is a directed path form $w$ to $v$. For example, suppose that there is a directed path from $e_+$ to $e'_+$, but no directed path from $e'_+$ to $e_+$. Since $D(A)$ is strongly connected, there is a directed path from $e'$ to $e$ and hence there must be  a directed path from $e'_+$ to $e_-$.  This implies there is a directed path from $e'_-$ to $e_+$. Also there is a directed path from $e_- $ to $e'_-$. Thus, we have constructed a directed path from $e'_+$ to $e_-$ to $e'_-$ to $e_+$, giving a contradiction.

The above discussion implies that  $D(\wh A)$ is either strongly connected or it has two 
components each of which is isomorphic to $D(A)$. In the later case we are done since $f$ is orientable. So we henceforth assume that $D(\wh A)$ is strongly connected, or equivalently that $\wh A$ is irreducible. 

Since $A$ is primitive
the gcd of the lengths of directed cycles of $D(A)$ is $1$  \cite[Theorem 3.4.4]{BrualdiRyser}. Since $D(\wh A) \to D(A)$ is a $2$-fold covering, the gcd of lengths of directed cycles of $D(\wh A)$ is at most $2$. If it is $1$, then $\wh A$ is also primitive and we are done. Otherwise the gcd is $2$ so that $\wh A^2$ is reducible and $D(\wh A^2)$ has two connected components.  Therefore, we must have that $f^2$ is pos-orientable by the argument above.

We will now show that $f$ is in fact neg-orientable. Notice that if $\wh A$ is not primitive, then every odd length cycle of $D(A)$ fails to lift to $D(\wh A)$. 
This means the edge path $f^k(e)$ can map over $e$ preserving orientation only when $k$ is even. Therefore $D(\hat{A})$ is bipartite with each directed edge going from one recurrent component of $D(\hat{A}^2)$ to the other.
Each of these recurrent components orients the graph map $f^2$ and with this orientation of $G$, $f$ maps each edge over a negatively oriented edge path. Hence, $f$ is neg-orientable.
\end{proof}

\subsection{Proof of \Cref{th:stretch_or}}
We now give the proof of \Cref{th:stretch_or} using the oriented edge-double construction.

\begin{proof} [Proof of \Cref{th:stretch_or}]
Let $f \colon G\to G$ be a primitive graph map and first
suppose $f$ is orientable. Recall $A$ is the primitive transition matrix of $f$ with spectral radius $\lambda:=\lambda_f>1$ and so there is an eigenvector $v\in \R^E$ such that $Av = \lambda v$. 
\Cref{fact:matrix} then implies that $v$ represents a $1$-chain (using our identification $C_1(G; \R)\cong  \R^E$) such that $M v = \pm \lambda v$, where the negative sign occurs exactly in the neg-orientable case.
In fact, $v \in H_1(G; \R) \le C_1(G;\R)$ since otherwise $\partial\colon C_1(G; \R) \to C_0(G; \R)$ maps $v$ to an eigenvector whose eigenvalue has absolute value $\lambda>1$, contradicting that $f$ acts on $C_0(G; \R)$ as a permutation.
This shows that $f_*(v) = \pm \lambda v$.
Since it is always the case that $\rho_f \le \lambda_f$, we conclude that 
$\lambda_f = \rho_f$ when $f$ is orientable. 

Next suppose that $f$ is not orientable. Then by \Cref{lem:prim_lift}, $\wh f \colon \wh G \to \wh G$ is orientable and has primitive transition matrix $\wh A$. Hence by the first paragraph above (and  \Cref{lem:prim_lift}) we have $\lambda := \lambda_f = \lambda_{\wh f} = \rho_{\wh f}$. By the Perron-Frobenius theorem, $\lambda$ is a simple eigenvalue of $\wh A$, i.e.~it has algebraic multiplicity $1$, and
 every other eigenvalue $\mu$ of $\wh A\colon H_1(\wh G)\to H_1(\wh G)$ has $\abs{\mu} < \lambda$. Moreover, there is a $\lambda$--eigenvector $\wh v\in H_1(\wh G)$ that lies in the positive cone of $C_1(\wh G)$ spanned by the 1-chains $e_+$ and $e_-$ for all $e \in EG$. We claim that $\wh v\in E_{-1}$, where recall $H_1(\wh G) = E_1 \oplus E_{-1}$. Indeed, since $\sigma$ commutes with $\wh f$, $\sigma_*(\wh v)$ is also a $\lambda$--eigenvector. Since this eigenspace is one dimensional $\sigma_*(\wh v) = \pm \wh v$. Now $\sigma_*$ sends the positive cone to the negative cone, therefore it must be that $\wh v\in E_{-1} = \ker(p_*)$ as claimed. From this, we see that any eigenvector of $\wh A$ in $E_1$ corresponding to an eigenvalue $\mu$ has  $\vert \mu \vert < \lambda$. By \Cref{lem:splitting}, $p_* \colon E_1 \to H_1(G)$ is a linear isomorphism that conjugates the action of $\wh f_*$ on $E_1$ to the action of $f_*$ on $H_1(G)$. We conclude that each eigenvalue $\mu$ of $f_*$ has $\vert \mu \vert < \lambda$. Hence, $\rho_f< \lambda_f$ and the proof is complete.

The moreover statement now follows from \Cref{fact:matrix}, since if $f$ is $\pm1$--orientable then $M = \pm A$ and so  $\pm \lambda_f$ is evidently a zero of the polynomial $\det(uI - f_*)$ in the variable $u$.
\end{proof}

\section{Orientability for free group endomorphisms}
\label{sec:endo}

In this section we introduce a notion of orientability for fully irreducible free group endomorphisms and characterize this property in terms of stretch factors,  train track representatives, and expanding laminations.

\subsection{Stretch factors for endomorphisms}
Let $\free$ be a finite-rank free group and $\varphi\colon \free\to \free$ any  endomorphism. The \define{(geometric) stretch factor} of $\varphi$ is defined as
\[
\lambda_\varphi = \sup_{w\in \free} \liminf_{n\to \infty} \sqrt[n]{\norm{\varphi^n(w)}},
\]
where the supremum is over all elements of $\free$ and $\norm{\cdot}$ is the length of the conjugacy class with respect to some (or any) fixed generating set of $\free$. The \define{homological stretch factor} of $\varphi$ is defined as the spectral radius for the action of $\varphi_*$ on $H_1(\free;\R)$ and is calculated as
\[
\rho_\varphi = \sup_{v\in H_1(\free,\R)} \liminf_{n\to \infty} \sqrt[n]{\norm{\varphi_*^n(v)}},
\]
where $\norm{\cdot}$ denotes some (or any) fixed norm on the vector space $H_1(\free;\R)$. 

Since the conjugacy length on $\free$ and  norm on $H_1(\free;\R)$ can be chosen such that $\norm{w^{\text{ab}}}\le \norm{w}$ for any element $w\in \free$ (where $w^{\text{ab}}$ denotes the image in the abelianization $H_1(\free;\Z)$ of $\free$), it is immediate that $\rho_\varphi\le \lambda_\varphi$ for any endomorphism. 

\subsection{Dual splittings}
\label{sec:dual_splittings}
Recall from \S\ref{sec:intro_BNS} that any endomorphism $\phi\colon B\to B$ of a group determines a generalized HNN extension $B \ast_\phi$ along with a dual class given by the projection $B \ast_\phi \to \Z$. We say that two such HNN extensions are \define{equivalent} if there is an isomorphism $B_1 \ast_{\phi_1}\cong B_2 \ast_{\phi_2}$ that respects the projections to $\Z$. 
Note that endomorphisms of distinct groups $B_1,B_2$ may determine equivalent HNN extensions and that, in principle, such endomorphisms $\phi_1,\phi_2$ may have distinct stretch factors. However, in the case of finite-rank free groups, we have the following relationships between equivalent splittings:

\begin{lemma} \label{lem:stretchdef}
Let $\varphi\colon \free\to \free$ be an endomorphism of a finite-rank free group.
\begin{itemize}
\item There exists an injective endomorphism $\phi\colon B\to B$ of a free group with $\mathrm{rank}(B)\le \mathrm{rank}(\free)$ such that $\free\ast_\varphi$ and $B\ast_\phi$ are equivalent.
\item If $\free\ast_\varphi$ is equivalent to $B \ast_\phi$ for some endomorphism $\phi\colon B\to B$ of a finite-rank free group, then $\lambda_\varphi = \lambda_\phi$ and $\rho_\varphi = \rho_\phi$. In fact, the characteristic polynomials of $\varphi$ and $\phi$ acting on homology agree up to multiplication by $t^k$.
\end{itemize}
\end{lemma}

\begin{proof}
The first facts are explained in \S2.4 of \cite{DKL2}. Specifically, Corollary 2.8 of that paper provides an injective endomorphism $\phi\colon B\to B$ as in the first item and Proposition 2.9 implies the invariance of geometric stretch factors in the second. The fact that $\rho_\varphi=\rho_\phi$ with equal characteristic polynomials up to a factor is a consequence of our determinant formula for computing the Alexander polynomial, which is an invariant of the group $\free\ast_\varphi$; see \Cref{lem:spec_char_poly} and \Cref{rmk:char_pol} below. 
\end{proof}

\subsection{Train track representatives}
\label{sec:tt_reps}
A \define{topological representative} of an endomorphism $\varphi\colon\free\to \free$ is a graph map $f\colon G\to G$ for which there is a choice of basepoint $v\in G$, path $\beta$ from $f(v)$ to $v$, and isomorphism $\free \cong \pi_1(G,v)$ that conjugates $\varphi$ to the endomorphism $\gamma\mapsto \bar{\beta} f(\gamma)\beta$ of $\pi_1(G,v)$. 
It follows easily from the definitions that if $f \colon G \to G$ represents $\varphi$, then $\lambda_\varphi\le \lambda_f$ 
 and $\rho_\varphi = \rho_f$. 

In part to obtain more precise control on geometric stretch factors, Bestvina and Handel \cite{BestvinaHandel-TTs} introduced the notion of a \define{train track map}, 
which is a graph map $f\colon G\to G$ such that for every edge $e$ and power $k\ge 1$ the edge path $f^k(e)$ does not backtrack.
The following connection to geometric stretch factors is well-known and is originally due to Bestvina--Handel in the case of automorphisms.
\begin{lemma}[{\cite[Proposition 2.2]{DKL2}}]
If $f$ is an irreducible train track representative of a free group endomorphism $\varphi$, then $\lambda_\varphi = \lambda_f$.
\end{lemma}

\subsection{Full irreducibility}
Recall that an endomorphism $\varphi\colon \free\to \free$ of a free group is \define{fully irreducible} 
if for any proper nontrivial free factor $A \le \free$, $\varphi^n(A)$ is not conjugate into $A$ for any $n\ge 1$.
Note that fully irreducible endomorphisms are necessarily injective (see, e.g., \cite[Lemma 2.2]{mutanguha2020irreducible}).

In light of \Cref{th:stretch_or}, it will be important to know that the train track maps we consider are automatically primitive. 
The following lemma is essentially well-known, at least in the automorphism case, but we provide a proof for completeness and because it is not always stated correctly in the literature.  

\begin{lemma} \label{lem:get_prim}
Suppose an endomorphism $\varphi \colon \free \to \free$ is either fully irreducible or represented by a pseudo-Anosov homeomorphism. Then $\varphi$ admits a train track representative $f\colon G\to G$. Moreover, the following are equivalent for any such representative: 
{\setlength{\multicolsep}{2pt}
\begin{multicols}{2}
\begin{enumerate}
\item $f$ has no invariant forests, 
\item $f$ is primitive,
\item $f$ is expanding, and 
\item $f$ is irreducible.
\end{enumerate}
\end{multicols}}
\end{lemma}

Note that any train track representative satisfies these conditions after collapsing the invariant forest. In the sequel, we usually refer to such train tracks as irreducible.

\begin{proof}
It is a seminal result of Bestvina and Handel that automorphisms which are fully irreducible \cite{BestvinaHandel-TTs} or represented by a pseudo-Anosov \cite[Theorem 4.1.3]{BestvinaHandel-surfaceTT} admit a train track representative. The existence for non-surjective endomorphisms was proven by Reynolds \cite{MR2949815} and independently Mutanguha \cite{mutanguha2020irreducible}.
It is clear that (2) $\implies$ (3) and that  (2) $\implies$ (4) $\implies$ (1).
To see (1) $\implies$ (2), suppose $f$ does not have an invariant forest and let $A$ be its transition matrix. To show that $A$ is primitive, it suffices to show that $A^n$ is irreducible for each $n\ge1$. Otherwise, there is $n\ge1$ and edges $e,e'$ of $G$ such that $(f^n)^j(e)$ does not cross $e'$ for any $j\ge 1$. 
Since there is no invariant forest, the lengths of the nonbacktracking paths $(f^n)^j(e)$ go to infinity in $j$. 
These 
iterates then fill a subgraph $G'$ of $G$ which is not a forest and for which $f^n(G') \subset G'$. 
Since $G'$ does not contain $e'$, it represents a proper free factor which is an immediate contradiction when $\varphi$ is fully irreducible. If instead $\varphi$ is induced by a pseudo-Anosov homeomorphism $h\colon S\to S$ of a punctured surface with $\pi_1(S)\cong \free$, then $G'$ must be a rank $1$ subgraph whose embedded core $\alpha$ represents a peripheral curve. Indeed, these peripheral curves are the only free factors that are invariant under a pseudo-Anosov. By construction, $\alpha$ is legal and hence its length grows under iteration. This is a contradiction to the fact that $h$ permutes the peripheral curves of $S$. 

The final implication (3) $\implies$ (4) is similar:
For any edge $e$ of $G$ the union of images $f^j(e)$ forms an $f$--invariant subgraph $G'\subset G$. 
Since $f$ is an expanding train track map, the nonbacktracking edge paths $f^j(e)$ become arbitrarily long. Hence $G'$ cannot be a forest, since that would force these long immersed paths $f^j(e)$ to embed in the finite graph $G'$. If $G'$ is a proper subgraph of $G$, then we obtain a contradiction by the argument above. This shows that $f$ is irreducible.
\end{proof}

\subsection{Orientable laminations}
\label{sec:or_lam}
Here we briefly review the theory of expanding laminations and attracting trees for free group endomorphisms. This allows for an intrinsically motivated definition of orientability and completes the analogy with the pseudo-Anosov theory.

In  \cite[Section 1]{bestvina1997laminations}, Bestvina--Feighn--Handel define a canonical expanding lamination $\mc L^+_\varphi$ for every fully irreducible automorphism $\varphi$ of $\free$. One may check that the \emph{auto}morphism hypothesis is only used in their argument to ensure the existence of an irreducible train track representative $f\colon G\to G$. Since endomorphisms also admit such representatives (\Cref{lem:get_prim}), the canonical construction of $\mc L^+_\varphi$ from \cite{bestvina1997laminations} in fact applies to any fully irreducible endomorphism $\varphi\colon \free \to \free$. 
The lamination is realized as a $\free$--invariant collection of properly embedded bi-infinite lines (called \define{leaves}) in the universal cover $\wt G$ that is invariant under the lifted train track map $\wt f \colon \wt G\to \wt G$.

The fully irreducible endomorphism $\varphi$ also comes equipped with a canonical \define{attracting $\R$--tree} $T^+_\varphi$ equipped with an isometric $\free$--action, which is the unique attracting fixed point for the action $\varphi$ on the closure of Culler--Vogtmann outer space; see \cite{bestvina1997laminations, levitt2003irreducible,MR2949815,mutanguha2020irreducible}.
The tree may be obtained from the train track representative $f$ via an explicit construction that provides an $\free$--equivariant, $1$--Lipschitz quotient map $F\colon \wt G \to T^+_\varphi$ that restricts to an isometry on each leaf of $\mc L^+_\varphi$. Further, $\varphi$ is represented by $\lambda_\varphi$--homothety $f^+\colon T^+_\varphi\to T^+_\varphi$ of this tree such that $F\circ \wt f = f^+ \circ F$; see \cite[\S2.5]{HM-AxesOutSpace}.

\begin{definition}
\label{def:orientable_lamination}
The expanding lamination $\mc L_\varphi^+$ of a fully irreducible endomorphism $\varphi$ is \define{orientable} if its leaves may each be $\free$--equivariantly oriented so that their images in the attracting tree $T_\varphi^+$ have compatible orientations whenever they overlap.
\end{definition}

\begin{remark}
This generalizes the notion of transverse orientability for the contracting
measured foliation $\mathcal{F}_-$ of a pseudo-Anosov surface homeomorphism $\psi$. Indeed, one may lift to the universal cover and collapse the leaves of $\wt{\mathcal{F}}_-$ to obtain an $\R$--tree $T_\psi^+$ in which each leaf of $\wt{\mathcal{F}}_+$ maps to a properly embedded bi-infinite line. This collection of lines in $T_\psi^+$ is orientable in the sense of \Cref{def:orientable_lamination} if and only if the contracting foliation $\mathcal{F}_-$ of $\psi$ is transversely orientable on the surface.
\end{remark}

\subsection{Characterizing orientability}
\label{sec:characterize_orientability}
We are now ready to establish \Cref{thm:orientable-auto} by proving that the various manifestations of orientability all agree for fully irreducible endomorphisms.

\begin{theorem} \label{th:stretch_or_auto}
For a fully irreducible free group endomorphism $\varphi\colon \free\to \free$, the following conditions 
are equivalent:
\begin{enumerate}
\item $\lambda_\varphi = \rho_\varphi$ with $+\lambda_\varphi$ (resp.\ $-\lambda_\varphi$) an eigenvalue of $\varphi_*$ acting on $H_1(\free)$,
\item $\varphi$ can be represented by a positively (resp.\ negatively) orientable graph map,
\item every irreducible train track representative of $\varphi$ is positively (resp.\ negatively) orientable,
\item the expanding lamination $\mc L^+_\varphi$ of $\varphi$ is orientable with the action of $\varphi$ on the attracting tree $T^+_\varphi$ preserving (resp.\ reversing) the orientations on the leaves of $\mathcal{L}^+_\varphi$.
\end{enumerate}
\end{theorem}

Accordingly, we say a fully irreducible endomorphism is \define{(positively/negatively) orientable} if it satisfies the equivalent conditions of \Cref{th:stretch_or_auto}. 
To prove the characterization in terms of laminations, we will use the following lemma:

\begin{lemma}
For any arc $J$ in $T^+_\varphi$, there is a finite chain of leaves $\ell_i$ of the expanding lamination $\mc L^+_\varphi$ such that $J \subset \ell_1\cup\dots\cup\ell_n$ with  each $\ell_i\cap \ell_{i+1}$ a nondegenerate arc (i.e it contains at least two points).
\end{lemma}
\begin{proof}
In the case that $\varphi$ is an automorphism, it is well-known that the action $\free \curvearrowright T^+_\varphi$ is \define{indecomposable} \cite[Theorem 2.1]{coulbois2012botany}. Following Guirardel \cite{guirardel2008actions}, this means that for any other nondegenerate arc $I \subset T^+_\varphi$, there are $g_1, \ldots g_k \in F_n$ such that $J \subset g_1I \cup \ldots \cup g_kI$ and $g_iI \cap g_{i+1}I$ is nondegenerate. Hence choosing $I$ to be any arc in a leaf of $\mc L^+_\varphi$ completes the proof.

When $\varphi$ is non-surjective, there is an irreducible train track $f\colon G\to G$ that is an immersion with connected Whitehead graphs \cite{mutanguha2020irreducible}. Thus by construction the attracting tree $T^+_\varphi$ is simply the universal cover $\wt G$ with an appropriate metric. Recall that the Whitehead graph of a vertex of $\wt G$ records which turns are taken by leaves of the lamination. Since these graphs are connected, for any vertex of $\wt G$ we may choose a chain of leaves $\ell_1,\dots,\ell_{k}$ with nondegenerate overlaps $\ell_i\cap \ell_{i+1}$ that connect any pair of directions at that vertex. Concatenating such chains for the finitely many vertices along the arc $J$ produces a chain covering $J$ as in the lemma.
\end{proof}

This has two important consequences when $\mc L^+_\varphi$ is orientable: First, the orientation on any leaf determines the orientation on all other leaves; thus the orientation is unique up to a global reversal, and the action of $f^+$ on $T^+_\varphi$ either preserves or reverses the orientation.
Second, any oriented arc $\alpha$ in $T^+_\varphi$ can be decomposed as a concatenation of subarcs of leaves of $\mc L^+_\varphi$.
Since leaves are consistently oriented, each subarc inherits a signed length from the metric on $T^+_\varphi$, and we may add to get a signed length $\int_\alpha d\mc L^+_\varphi$ for $\alpha$. This assignment $\alpha \mapsto \int_\alpha d\mc L^+_\varphi$ is $\free$--equivariant and satisfies
$\int_{f^+(\alpha)} d \mc L^+_\varphi = \pm \lambda_\varphi \cdot \int_{\alpha} d\mc L^+_\varphi$, where the sign is positive or negative depending on whether the $\lambda_\varphi$--homothety $f^+$ preserves or reverses the orientation on $\mc L^+_\varphi$.

\begin{lemma} \label{prop:cocyle}
If $\mc L = \mc L^+_\varphi$ is orientable and $f^+$ preserves (resp.\ reverses) its orientation and fixes the point $x \in T^+_\varphi$,
then the homomorphism $\xi \colon \free \to \R$ determined by
\[
 g \mapsto \int_{[x,gx]} d \mc L\qquad\text{for }g\in \free
\]
defines a $\lambda_\varphi$-eigenvector (resp.\ $- \lambda_\varphi$-eigenvector) of $\varphi^*$ in $H^1(\free; \mathbb{R})$.
\end{lemma}

\begin{proof}
First note that by considering the corresponding tripod in $T^+_\varphi$, we see that  for any $w,y,z \in T^+_\varphi$,
\[
\int_{[w,z]} d\mc L = \int_{[w,y]} d\mc L + \int_{[y,z]} d\mc L.
\]
Using $\free$--equivariance, this translates into the fact that for any $g,h \in \free$, $\xi(gh) = \xi(g)+\xi(h)$ by setting $w=x, \: y= gx, \: z= ghx$. To complete the proof, observe that
\[
\xi(\varphi(g)) = \int_{[x,f^+(gx)]} d\mc L = \int_{f^+([x,gx])} d\mc L = \pm \lambda_\varphi \cdot \xi(g),
\]
as required.
\end{proof}

\begin{proof}[Proof of \Cref{th:stretch_or_auto}]
By \Cref{lem:get_prim} every irreducible train track representative of $\varphi$ 
 is primitive. Hence, the equivalence of (1), (2), and (3) follows immediately from \Cref{th:stretch_or}. Since every subarc of a leaf of $\mc L^+_\varphi$  in $\wt G$ is contained in $\wt f^n(e)$ for some edge of $\wt G$, it is clear that $f$ being orientable allows one to consistently and equivariantly  orient the leaves of $\mc L^+_\varphi$ in $T^+_\varphi$; thus (3) implies (4). Finally, \Cref{prop:cocyle} shows that (4) implies (1).
\end{proof}

\section{Computing the Alexander polynomial}
\label{sec:alex_poly_of_free-by-cyclic}
In this section, we prove \Cref{thm:det_form} which gives a determinant formula for the Alexander polynomial of the mapping torus of a graph map.  See \Cref{sec:AntiO} for an explicit computation.

\subsection{Fitting ideals and invariants} \label{sec:invariants}
Let $R$ be a unique factorization domain (UFD).
Throughout, we will use the notation $p \doteq q$ to mean that two elements $p,q$ of $R$ are equal up to multiplication by a unit.
Let $N$ be a finitely generated
$R$-module. If $N$ has a presentation,
\[
 R^m \overset{M}{\longrightarrow} R^n \to N \to 0,
\]
then the \emph{$i$th Fitting ideal} (also called the $i$th determinantal ideal)
$\mathrm{Fit}_i(N) \subset R$ is the ideal generated by the $(n-i) \times (n-i)$ minors of $M$. It is a fact that these ideals do not depend on the choice of presentation.

Recall that since $R$ is a UFD the greatest common divisor (gcd) of elements of $R$ is well-defined up to a unit in $R$. For an $R$-module $N$ as above, the gcd of the elements of 
$\mathrm{Fit}_i(N)$,
$\mathrm{gcd}(\mathrm{Fit}_i(N))$, is called the \emph{$i$th Fitting invariant} of $N$. 

Let $H$ be a finitely generated free abelian group. Then the group ring $\Z[H]$ is a UFD and its \emph{augmentation ideal} $\mc A$ is defined to be the kernel of the augmentation homomorphism $\mathbb{Z}[H] \to \mathbb{Z}$ defined by extending the assignment $h \mapsto 1$ for $h \in H$.

Finally, note that any $P \in \Z[H]$ can be written 
in the form $P = \sum_{h\in H} a_h h$, where each $a_h\in \Z$ and $a_h = 0$ for all but finitely many $h$. Such a $P \in \Z[H]$ is said to be \emph{symmetric} if, up to a unit in $\Z[H]$, it is equal to $\sum_{h\in H} a_h h^{-1}$.
In general the set $\supp(P) =  \{h \in H : a_h \ne 0\}$ is called the \define{support} of $P$.
The \emph{specialization} of $P$ at a class $u \in  \mathrm{Hom}(H,\Z)$ is the the single-variable Laurent polynomial
\[
P^u(t) = \sum_{h\in H} a_h t^{u(h)}.
\]

\subsection{Alexander polynomial generalities}
\label{sec:alex_poly}
Let $X$ be a connected cell complex with $\Gamma = \pi_1(X)$ and let $H = H_1(X)/\text{torsion}$.

Let $\wt X^{ab}$ be the universal free abelian cover of $X$, i.e.~the cover corresponding to the kernel of $\Gamma \to H$ whose deck group is $H$. Fix $x \in X^{(0)}$ and let $\wt x$ be its full preimage in $\wt X^{ab}$. The homology group $H_1(\wt X^{ab}, \wt x)$ naturally has the structure of a $\mathbb{Z}[H]$--module and is called the \emph{Alexander module}. Its first Fitting invariant is called the \emph{Alexander polynomial} and is denoted by $\Delta_X$. In other words,
\[
\Delta_X \doteq \text{gcd} \: \mathrm{Fit}_1(H_1(\wt X^{ab}, \wt x)),
\]
which is an element of $\mathbb{Z}[H]$, defined up to a unit $\pm h \in \mathbb{Z}[H]$. We note that if $T$ is a maximal tree in $X^{(1)}$ and $\wt T$ is its full preimage in $\wt X^{ab}$, then we see that the $\mathbb{Z}[H]$--modules $H_1(\wt X^{ab}, \wt x)$ and $H_1(\wt X^{ab}, \wt T)$ are isomorphic by considering the long exact sequence associated to the triple $(\wt X^{ab}, \wt T, \wt x)$. 

In fact, the Alexander polynomial of $X$ depends only on its fundamental group (see \cite[Section 2]{McMullen--Alexander_polynomial}), and so in what follows we also write $\Delta_\Gamma$ to denote the Alexander polynomial $\Delta_X$, where $\Gamma = \pi_1(X)$. 

\subsection{Mapping tori and homology}
\label{sec:mapping_tori}
Let $f\colon G\to G$ be a graph map on a finite graph $G$. 
Fix an arbitrary orientation on each edge of $G$ and let $V = VG$ and $E = EG$ respectively denote the sets of vertices and (oriented) edges of $G$. Let $X := X_f$ 
be the mapping torus with the usual cell structure, so that $\pi_1(X) = \pi_1(G)\ast_{f}$ is the associated generalized HNN-extension. 

In a bit more detail: Give $G \times [0,1]$ the cell structure induced from the product
cell structure so that the oriented $2$-cells induce the given orientation on the edges of $G \times \{0\}$ and the opposite orientation on the edges of $G \times \{1\}$. Then subdivide $G \times \{1\}$ so that $f \colon G \times \{1\} \to G \times \{0\}$ maps open cells homeomorphically onto their images.
The mapping torus $X$ is then the quotient
\[
X = G \times [0,1] / (x,1) \sim (f(x),0),
\]
with its induced cellular structure. See \Cref{fig:2cell}.

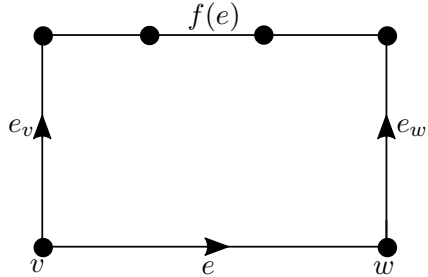
\begin{figure}[ht]
    \centering
    \def\svgwidth{.37\columnwidth}
\begingroup%
  \makeatletter%
  \providecommand\color[2][]{%
    \errmessage{(Inkscape) Color is used for the text in Inkscape, but the package 'color.sty' is not loaded}%
    \renewcommand\color[2][]{}%
  }%
  \providecommand\transparent[1]{%
    \errmessage{(Inkscape) Transparency is used (non-zero) for the text in Inkscape, but the package 'transparent.sty' is not loaded}%
    \renewcommand\transparent[1]{}%
  }%
  \providecommand\rotatebox[2]{#2}%
  \newcommand*\fsize{\dimexpr\f@size pt\relax}%
  \newcommand*\lineheight[1]{\fontsize{\fsize}{#1\fsize}\selectfont}%
  \ifx\svgwidth\undefined%
    \setlength{\unitlength}{183.2382517bp}%
    \ifx\svgscale\undefined%
      \relax%
    \else%
      \setlength{\unitlength}{\unitlength * \real{\svgscale}}%
    \fi%
  \else%
    \setlength{\unitlength}{\svgwidth}%
  \fi%
  \global\let\svgwidth\undefined%
  \global\let\svgscale\undefined%
  \makeatother%
  \begin{picture}(1,0.60892216)%
    \lineheight{1}%
    \setlength\tabcolsep{0pt}%
    \put(0,0){\includegraphics[width=\unitlength,page=1]{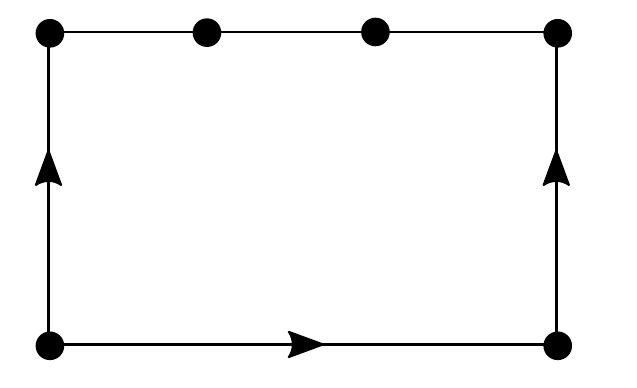}}%
    \put(0.89775415,0.33350171){\color[rgb]{0,0,0}\makebox(0,0)[lt]{\lineheight{1.25}\smash{\begin{tabular}[t]{l}$e_w$\end{tabular}}}}%
    \put(-0.00271803,0.33453537){\color[rgb]{0,0,0}\makebox(0,0)[lt]{\lineheight{1.25}\smash{\begin{tabular}[t]{l}$e_v$\end{tabular}}}}%
    \put(0.44586905,0.00481251){\color[rgb]{0,0,0}\makebox(0,0)[lt]{\lineheight{1.25}\smash{\begin{tabular}[t]{l}$e$\end{tabular}}}}%
    \put(0.04752034,0.00916007){\color[rgb]{0,0,0}\makebox(0,0)[lt]{\lineheight{1.25}\smash{\begin{tabular}[t]{l}$v$\end{tabular}}}}%
    \put(0.84338657,0.0048177){\color[rgb]{0,0,0}\makebox(0,0)[lt]{\lineheight{1.25}\smash{\begin{tabular}[t]{l}$w$\end{tabular}}}}%
    \put(0.41279229,0.5840442){\color[rgb]{0,0,0}\makebox(0,0)[lt]{\lineheight{1.25}\smash{\begin{tabular}[t]{l}$f(e)$\end{tabular}}}}%
  \end{picture}%
\endgroup%

    \caption{A $2$-cell of $X_f$ with indicated orientations on edges.}
    \label{fig:2cell}
\end{figure}

The mapping torus comes with a natural map $X \to S^1$. In the cell structure we view edges of $G$ as ``horizontal'' and the other edges as ``vertical,'' oriented so that each maps to a positively oriented based loop in $S^1$ representing the generator $1\in \Z$. Let $\mathcal{E}$ denote the set of (oriented) edges in $X$.
Each vertex $v\in V$ is the tail (i.e., initial vertex) of a unique vertical edge $e_v$ in $X$. This gives a bijection 
$v\mapsto e_v$ which we use to identify $V$ with the set of vertical edges of $X$. 
When we wish to be explicit, we denote this bijection by $\nu$.
Similarly each edge $e\in E$ gives rise to an oriented $2$--cell $\sigma_e$ of $X$ whose boundary crosses $e$ positively and the image path $f(e)$ negatively. In this way we see that the $1$--cells of $X$ are in bijective correspondence with $E\sqcup V$, and the $2$--cells with the edges $E$ of $G$. With these identifications, the $\Z$--modules of $0$--chains in $X$ are $C_0(X) = C_0(G)\cong \Z[V]$, the $1$--chains are $C_1(X)\cong \Z[\mathcal{E}] \cong\Z[E\sqcup V]$, and the $2$--chains are $C_2(X) \cong \Z[E]$.

The map $f\colon G\to G$ restricts to a map $V\to V$ which we express as an integer $V\times V$ matrix $P$ whose $(v,w)$ entry equals $1$ if $f(w) = v$ and is zero otherwise. That is, the map that $f$ induces on $0$--chains $C_0(G) \cong \Z[V] = \Z^V$ is simply the linear map $\Z^V\to \Z^V$ given by $x\mapsto Px$. Similarly, the map that $f$ induces on $C_1(G)\cong \Z^E$ is expressed as an integer $E\times E$ matrix $M$ whose $(e,e')$ entry is the number of times the edge path $f(e')$ crosses $e$, counted with signs.
The boundary maps on chains may then be written as
\[\begin{tikzcd}
C_2(X) \arrow[rrr,"\partial_2"] \arrow[d,"\cong"] &&& C_1(X) \arrow[rrr,"\partial_1"] \arrow[d,  "\cong"] &&& C_0(X) \arrow[d, "\cong"]\\
\Z^E \arrow[rrr, "\begin{pmatrix}\pi_E \partial_2 \\ \pi_V\partial_2\end{pmatrix}"]&&& \Z^E\oplus \Z^V \arrow[rrr, "\begin{pmatrix}\partial_1\vert_E \;\;\; \partial_1\vert_V\end{pmatrix}"] &&& \Z^V
\end{tikzcd}
\]
That is, $\partial_1\vert_E\colon \Z^E\to \Z^V$ is the restriction of $\partial_1$, viewed as an $E\times V$ matrix, and similarly for $\partial_1\vert_V$. Likewise $\pi_E\partial_2$ is the $E\times E$ matrix describing the projection $\Z^E\stackrel{\partial_2}{\longrightarrow} \Z^E\oplus \Z^V \stackrel{\pi_E}{\longrightarrow}\Z^E$, and similarly for $\pi_V\partial_2$. 
Note that the boundary of a vertical edge $e_v$ corresponding to its tail vertex $v\in V$, is $f(v) - v$. Hence $\partial_1\vert_V = P-I$. Similarly, if an edge $e\in E \subset C_1(X)$ has image $f(e) = a_1 e_1 + \dots + a_k e_k\in C_1(X)$, then the corresponding $2$--cell has boundary
\[\partial_2(\sigma_e) = e  - a_1 e_1 - \dots - a_k e_k+ \nu \circ \partial_1(e).\]
That is to say we have $\pi_E \partial_2 =I - M$, and $\pi_V \partial_2 = \partial_1\vert_E$.

\subsection{The universal free abelian cover of $X$}
\label{sec:abelian_cover}
To calculate the Alexander polynomial for $\pi_1(X)$,  let $H = H_1(X)/\text{torsion}$ and as usual let $\wt{X}^{ab}\to X$ be the universal free abelian cover with deck group $H$. In what follows, 
we abbreviate $\wt X^{ab}$ by $\wt X$. Again, we view lifts of edges of $G$ as being horizontal in $\wt{X}$, and lifts of the oriented vertical edges in $X$ as being oriented upward.

Let $\wt{G}\subset \wt{X}$ be the full preimage of $G$, and choose some connected component $\wt G_0$. Then $\wt G_0 \to G$ is the free abelian cover with deck group $K$, where $K$ is the image of the homomorphism $\pi_1(G)\to \pi_1(X)\to H$. 
Alternatively, 
$K$ can be intrinsically characterized as the quotient of $H_1(G)$ isomorphic to $\mathrm{Hom}(H^1(G)^f, \mathbb{Z})$ by duality. Here $H^1(G)^f$ is the cohomology of $G$ that is fixed by $f$ (see \cite[Section 8]{DKL2}).

The graph map $f \colon G \to G$ lifts to $\wt G_0$ and we 
fix such a lift 
\[
\wt f \colon \wt G_0 \to \wt G_0
\]
once and for all. This lift determines a splitting 
$H = K \oplus \langle \dt \rangle$ where $\dt$ is the deck transformation chosen so that for any $\wt x \in \wt G_0$, $\wt f(\wt x)$ is the first point in $\wt G_0$ encountered by 
the lift of the arc $x \times [0,1]$
starting at $\dt \wt x$. 
Here $x$ is the image of $\wt x$ under the covering $\wt X \to X$. Note that the splitting $H = K \oplus \langle \dt \rangle$ depends on our choice of $\wt f$ and  $\dt$ has been chosen to map to $-1$ under $\pi_1(X) \to  \pi_1(S^1)$. Hence, $\dt$ translates downward as a deck transformation of $\wt{X}$. For notational convenience, we set $\posgen = \dt^{-1}\in H$. 
For each $i\in \Z$ let $\wt G_i = \posgen^{i} \wt G_0$; this decomposes $\wt G$ into its connected components $\wt G = \bigsqcup_{i} \wt G_i$, with $\wt G_{i+1}$ situated above $\wt G_i$. 

For each cell in $G$, we fix arbitrary lifts in $\wt X$ as follows. For each $v\in V$ we fix a lift $\wt{v}\in \wt G_0$ and we let $\wt{e}_v$ denote the lift of $e_v$ with tail $\wt{v}$. For each edge $e\in E$ we fix, once and for all, a lift $\wt{e}\in \wt G_0$ which may or may not be based at the chosen lifts of vertices of $G$. We then let $\wt{\sigma}_e$ be the lift of the $2$--cell $\sigma_e$ whose boundary crosses $\wt{e}$ positively and $\posgen\wt{f}(\wt e)\subset \wt G_1$ negatively.  Note that this chooses a preferred lift to $\wt X$ for each edge of $\mc E$ which we call the \emph{base lifts}.
Since $H$ acts as deck transformations on $\wt X$, the $\Z$--modules of chains in $\wt X$ are naturally free $\Z[H]$--modules. Using our chosen base lifts as basis elements thus gives $\Z[H]$--module isomorphisms
\[C_2(\wt X) \cong \Z[H]^E, \quad C_1(\wt X) \cong \Z[H]^{E}\oplus \Z[H]^V,\quad\text{and}\quad C_0(\wt X) \cong \Z[H]^V.\]
The action of $K$ on $\wt G_0$ by deck transformations similarly gives $C_1(\wt G_0) \cong \Z[K]^E$ and $C_0(\wt G_0) \cong \Z[K]^V$. With  these identifications, the maps that $\wt{f}$ induces on $C_1(\wt G_0)$ and $C_0(\wt G_0)$ are expressed by $\Z[K]$--matrices $\wt M$ and $\wt P$ that respectively specialize to $M$ and $P$ under the ring morphism $\Z[K] \to \Z$ induced by the trivial group homomorphism $K \to \{1\}$.

\subsection{The chain complex of $\wt X$}
We now examine the boundary maps $\wt{\partial_i}$ for the cellular homology of $\wt X$. Firstly, let us write $\wt B = \wt{\partial_1}\vert_{\Z[H]^E}$ for the matrix representing the restriction of $\wt{\partial_1}$ to the horizontal edges of $\wt X$. Since the basis horizontal edges $\wt{e}$ and their boundary vertices lie in $\wt G_0$, we see that $\wt B$ is in fact a $\Z[K]$--valued $V\times E$ matrix. With this notation, each basis $2$--cell $\wt{\sigma}_e$ and each basis vertical $1$--cell $\wt{e}_v$ has boundary
\[\wt\partial_2(\tilde{\sigma}_e) = \wt{\nu}\circ \wt{\partial_1}(\wt{e}) + \wt{e} - \posgen\wt{f}(\wt{e}) = \wt B \wt{e} + \wt{e} - \posgen\wt M \wt{e}
\qquad\text{and}\qquad
\wt \partial_1(\wt{e}_v) = \posgen\wt P \wt{v}  - \wt{v}.\] 
Therefore the chain complex that computes the cellular homology of $\wt X$  takes the form
\[\begin{tikzcd}
C_2(\wt X) \arrow[rrrr,"\wt{\partial_2}"] \arrow[d,"\cong"] &&&& C_1(\wt X) \arrow[rrr,"\wt{\partial_1}"] \arrow[d,  "\cong"] &&& C_0(\wt X) \arrow[d, "\cong"]\\
\Z[H]^E \arrow[rrrr, "\begin{pmatrix}I - \posgen\wt M\\ \wt B\end{pmatrix}"]&&&& \Z[H]^E\oplus \Z[H]^V \arrow[rrr, "\begin{pmatrix}\wt B \;\;\;\; \posgen\wt P - I\end{pmatrix}"] &&& \Z[H]^V
\end{tikzcd}
\]

The chain condition $\wt{\partial_1} \wt{\partial_2} = 0$ has a useful interpretation in terms of the coboundary maps $\wt{\delta^1}\colon C^0(\wt X) \to C^1(\wt X)$ and $\wt{\delta^2}\colon C^1(\wt X)\to C^2(\wt X)$ on cochains, where our bases induce isomorphisms $C^i(\wt X)\cong C_i(\wt X) \cong \Z[H]^{d_i}$ with respect to which  $\wt{\delta^i} = \wt{\partial_i}^t$ is simply the transpose of the boundary matrix.
For any vertex $w$ in $\wt X$, the vector $\wt{\delta^1}(w)\in C^1(\wt X)$ is the sum of those edges of $\wt X$ that terminate at $w$ minus the sum of edges that start at $w$.  
Similarly, for each edge $\eta$ of $\wt X$, the vector $\wt{\partial^2}(\eta)\in C^2(\wt X)$ is the sum of $2$--cells whose boundary crosses $\eta$ positively minus those whose boundary crosses $\eta$ negatively. For short hand, let us set $R_\eta = \wt{\delta^2}(\eta)\in C^2(\wt X)\cong \Z[H]^E$. Thus the row of the matrix $\wt{\partial_2}$ corresponding to an index $x\in \mathcal{E}$ is precisely $R_{\wt x}$. 
Now for any cochain $w\in C^0(\wt X)$, we may express $\wt{\delta^1}(w)\in C^1(\wt X)$ as a vector in $\Z[H]^\mathcal{E}$ and apply $\wt{\delta^2}$ to obtain a $\Z[H]$--linear dependence relation between the rows of $\wt{\partial_2}$:
\begin{equation}
\label{eqn:vertex-row-relation}
\wt{\delta^1}(w) = \sum_{x\in \mathcal{E}}a_x \wt x
\quad\implies\quad
 0 = \wt{\delta^2}(\wt{\delta^1}(w)) = \sum_{x\in \mc E} a_x R_{\wt x}.
\end{equation}

\subsection{Relative homology and vertex cycles}
\label{sec:rel_hom}
We will use relative homology to calculate the Alexander polynomial, and for this it will be convenient to use a well-chosen maximal tree in $X^{(1)}$. Let $\mathcal{V}\subset X$ be the union of all vertical edges, decomposed into its connected components $\mathcal{V} = \mathcal{V}_1\sqcup\dots\sqcup\mathcal{V}_m$. Note that each $\mathcal{V}_k$ is a directed embedded \emph{vertex cycle}, denoted $\mathcal{V}_k'$, with trees hanging off. Here the loop corresponds to a periodic orbit $v,f(v),\dots,f^{\ell_k}(v) =v$ of length $\ell_k\ge 1$ for the action of $f$ on $V = VG$ and the hanging trees correspond to other vertices $w\in V$ that map into the orbit $\{f(v),\dots, f^{\ell_k}(v)\}$. Let $c_k = [\mathcal{V}_k] = [\mathcal{V}_k']\in H$ denote the homology class defined by the cycle in $\mathcal{V}_k$. 

\begin{lemma} \label{lem:detP}
We have $\det(I - \posgen\wt P) \doteq \prod_{i=1}^m \left(1-c_i \right)$.
\end{lemma}

In \S\ref{sec:vertex_poly}, this formula will be used to define the vertex polynomial.

\begin{proof}
The components $\mathcal{V}_1\sqcup\dots\sqcup\mathcal{V}_m$ determine a block diagonal decomposition of $I-\posgen \wt P$ into blocks $I-\posgen \wt P_i$ for $1\le i\le m$ and so it suffices to show that $\det(I-\posgen \wt P_i) = 1-c_i$. As mentioned above, $\mc V_i$ is an embedded loop $\mc V_i'$ along with trees hanging off. If we index the vertices so that they are ordered in a way that is compatible with their distance to $\mc V_i'$, then it follows that $\det(I-\posgen \wt P_i) = \det(I-\posgen \wt P_i')$, where $\wt P_i'$ is the submatrix of $\wt P_i$ associated to the vertices of $\mc V_i'$.

Now let $v$ be a vertex in $\mc V_i'$ 
and 
suppose that the vertical loop $\mc V_i'$ has the 
vertex sequence $v = v_0, v_1, \ldots, v_{\ell_k} = v$. 
If we lift the vertical loop $\mc V_i'$ to $\wt X$ starting at $\wt v$, we obtain a path from $\wt v$ to $c_i \wt v$. The vertex sequence of this lift is
\[
 \wt v = \wt v_0, g_1\wt v_1, \ldots, g_{\ell_k}\wt v_{\ell_k} = c_i \wt v,
 \]
 where $\wt v_i$ is the base lift of $v_i$ and $g_i \in H$. 
By our choice of lift $\wt f \colon \wt G_0 \to \wt G_0$ and splitting $H = K \oplus \langle \dt \rangle$, it follows that $g_{j+1}\wt v_{j+1} = \posgen\wt P(g_j \wt v_j)$. So if we set $x_j = g_j^{-1}g_{j+1}$, then
$x_j \wt v_{j+1} = \posgen\wt P(\wt v_j)$. With this, and the fact that the matrices $\wt P$ and $\wt P_i'$ agree on the base lifts of vertices in $\mc V_i'$, an easy computation now gives the needed equality
\[
\det(I-\posgen \wt P_i') = 1 - \prod_{j=0}^{\ell_k-1} x_j = 1-c_i.\qedhere
\] 
\end{proof}

Choose a maximal tree $T_k$ inside of each component $\mathcal{V}_k$; since $\mathcal{V}_k$ is a loop with trees hanging off, this amounts to deleting exactly one edge $\eta_k$ so that $\mathcal{V}_k = T_k\cup \eta_k$. The union $T_1\sqcup \dots \sqcup T_m$ is then a forest in $X$, and we extend this to a maximal tree $T\subset X$ by adding horizontal edges that connect distinct components. 
We then use $\mathcal{E}\setminus T$ to denote the set of edges of $X$ that are not contained in $T$ and note that $\abs{\mathcal{E}\setminus T} = \abs{E}+1$.

Let $\wt T$ be the full preimage of $T$ in $\wt X$. 
The $\Z[H]$--module of relative $1$--chains for the pair $(\wt X, \wt T)$ is identified with $C_1(\wt X, \wt T) \cong \Z[H]^{\mathcal{E}\setminus T}$ 
and the chain complex for relative homology becomes
\[
\begin{tikzcd}
C_2(\wt X, \wt T)\cong \Z[H]^{E}  \arrow[r,"A"] & C_1(\wt X,\wt T)\cong \Z[H]^{\mathcal{E}\setminus T}  \arrow[r,] & C_0(\wt X,\wt T) = 0
\end{tikzcd}
\] 
where $A$ is simply the $(\mathcal{E}\setminus T)\times E$ matrix obtained from $\wt \partial_2$ by deleting the rows corresponding to edges in $T$. It follows that  the $\Z[H]$--module $H_1(\wt X, \wt T)$ may be presented as
\begin{align}\label{eq:def_A}
\Z[H]^E \stackrel{A}{\longrightarrow} \Z[H]^{\mathcal{E}\setminus T} \longrightarrow H_1(\wt X, \wt T) \to 0.
\end{align}
The Alexander polynomial $\Delta_{\Gamma}\in \Z[H]$ is thus the gcd of the $E\times E$ minors of $A$. We emphasize that here and throughout all identifications are made using the base lifts fixed in \S\ref{sec:abelian_cover}.

Since $H$ acts freely on the connected components of $\wt T$, fixing a component  $\wt T'$ gives an isomorphism $H_0(\wt T) \cong \Z[H]$ as $\Z[H]$--modules. Let $\alpha\colon \Z[H]^{\mathcal{E}\setminus T}\to \Z[H]$ be the composition

\begin{align} \label{eq:def_alpha}
 \Z[H]^{\mathcal{E}\setminus T} \cong C_1(\wt X, \wt T)\longrightarrow H_1(\wt X,\wt T)\stackrel{\sigma}{\longrightarrow} H_0(\wt T) \cong \Z[H],
 \end{align}
where $\sigma$ is the connecting map for the long exact sequence of the pair $(\wt X, \wt T)$. Note that if $e$ is any edge of $\mathcal{E} \setminus T$ and $\wh e$ is its (nonbase) lift to $\wt X$ whose tail is contained in $\wt T'$, then $\alpha(\wh e) = g-1$, where $g$ is represented by the loop in $e \cup T$ crossing $e$ positively. From this, we see that the image of $\alpha$ is the usual augmentation ideal $\mc A \subset \mathbb{Z}[H]$.

\subsection{The fundamental relations and the main lemma}
For each $i = 1,\dots, m$, let $w_i\in C^0(\wt X)$ be the sum of the vertices of $\wt T_i'$, which is the lift of $T_i$ contained in the fixed component $\wt T'$. 
The coboundary of $w_i$ then has the form $\wt \delta^1(w_i) = \sum_{x\in \mathcal{E}} b_x^i \wt x$ for some ring elements $b_x^i\in \Z[H]$. Let us calculate these coefficients $b_x^i$ for $x\in \mathcal{E}$.

\emph{Vertical edges:} 
Recall the set of vertical edges $\mc V$ is a disjoint union  $\mc V_1 \sqcup \cdots \sqcup \mc V_m$. If $x \in \mc V_k$ for some $k \neq i$, then each lift $\hat x$ of $x$ is disjoint from $\wt T_i'$ and hence has $\wt \delta^1(w_i)(\hat x) = 0$; thus $b^i_x = 0$ in this case. If $x \in \mc V_i$, then it is either in $T_i$ or is equal to $\eta_i$. In the first case, there is exactly one lift $\hat x$ that intersects $\wt T_i'$, but this lift has \emph{both} its tip and tail in $w_i$ meaning $\wt \delta^1(w_i)(\hat x) = w_i(\wt \partial_1(\hat x)) = 0$. Hence $b_x^i = 0$ for each vertical edge $x$ in $T_i$.

For $x = \eta_i$, there are precisely two lifts incident on $\wt T_i'$.
Writing $\wt \eta_i$ for the base lift of $\eta_i$, there is one translate, say $g \wt \eta_i$, whose tip lies in $\wt T_i'$ 
and another whose tail lies in $\wt T_i'$. Recall that $c_i\in H$ is the homology class of the unique nontrivial loop in $T\cup \eta_i$ that crosses $\eta_i$ once positively. Lifting this loop to $\wt X$, we see that the tip of $g\wt \eta_i$ lying in $\wt T_i'$ implies  the tail of $g \wt \eta_i$ lies in $c_i\inv \wt T'$. Thus $c_i g \wt \eta_i$ is the unique lift of $\eta_i$ whose tail lies in $\wt T_i'$. Since $\wt \delta^1(w_i)$ evaluates to $1$ on $g \wt \eta_i$ and to $-1$ on $c_ig \wt \eta_i$, we conclude that the coefficient of $\wt \eta_i$ in $\wt \delta^1(w_i)$ is $b_{\eta_i}^i = g(1-c_i)$. Notice that by definition of $\alpha\colon C_1(\wt X,\wt T)\to H_0(\wt T)\cong \Z[H]$ we have $\alpha(g \wt \eta_i) =1 - c_i\inv$ since the tip and tail of $g \wt \eta_i$ lie in $\wt T'$ and $c_i\inv \wt T'$. Therefore we have $b_{\eta_i}^i = -g^2 c_i\inv\alpha(\wt \eta_i) =  -g_i \alpha(\wt \eta_i)$ for some unit $g_i\in \Z[H]$.

\emph{Horizontal edges:} First let $x$ be a horizontal edge in the maximal tree $T$. If both end points both of $x$ lie outside $T_i$, then $b^i_x = 0$ since each lift $\hat x$ of $x$ is disjoint from $\wt T_i'$ and hence has $\wt \delta^1(w_i)(\hat x) = 0$. If instead $x$ has exactly one end point in $T_i$, 
then $x$ has exactly one lift $\hat x$ that intersects $\wt T_i'$, and this lift has only one endpoint in $\wt T_i'$; hence $b_x^i\in \Z[H]$ is a unit. Let $D_i$ denote the set of horizontal edges in $T$ that have exactly one end point in $T_i$. Finally, the nature of the coefficient $b_x^i$ for $x\in E\setminus T$ will not be relevant for us.

Combining these observations, we see that
\[\wt \delta^1(w_i) = -g_i \alpha(\wt \eta_i)  \wt \eta_i + \sum_{x\in D_i \cup (E\setminus T)} b_x^i \wt x\]
where $b_x^i$ is a unit in $\Z[H]$ for each $x$ in $D_i$ and $g_i$ is also a unit. The Observation \eqref{eqn:vertex-row-relation} therefore gives us a relation:
\begin{equation}
\label{eqn:fund_relation_for_nontree_vert_edges}
g_i \alpha(\wt \eta_i) R_{\eta_i} =  \sum_{x\in D_i \cup E\setminus T} b_x^i R_{x} \quad\text{ for each }i = 1,\dots, m.
\end{equation}
That is: For each vertical edge $\eta_i$ in $\mathcal{E}\setminus T$, multiplying the corresponding row $R_{\eta_i}$ of $A$ by $g_i \alpha(\wt \eta_i)$ produces a linear combination of the rows $R_x$ for horizontal $x$ in $\mathcal{E}\setminus T$ plus a linear combination \emph{with unit coefficients} of the rows $R_x$ of $\wt\partial_2$ corresponding to $x$ in $D_i$. 

\begin{lemma} \label{lem:dets}
Let $A_k$ denote the matrix obtained from $A$ (from equation (\ref{eq:def_A})) by removing the row for the vertical edge $\eta_k\in \mathcal{E}\setminus T$. Then $\det(A_k)\det(I-\posgen \wt P) \doteq \alpha(\wt \eta_k)\det(I - \posgen \wt M)$.
\end{lemma}
\begin{proof}
We need to relate the determinant of $A_k$ to that of $I - \posgen \wt M$, which is the submatrix of $\wt{\partial_2}$ corresponding to the horizontal edges $E$. The matrix $A_k$ already contains all rows $R_x$ of $\wt{\partial_2}$ corresponding to horizontal edges $x$ in $E\setminus T$. However, $A_k$ is missing the rows corresponding to horizontal edges in $T$ and in their place has extra rows corresponding to vertical edges of $\mathcal{E}\setminus T$ of the form $\eta_i$ except $\eta_k$. We will recursively use the fundamental relation \eqref{eqn:fund_relation_for_nontree_vert_edges} to replace each row for a vertical edge of $A_k$ with the row for a missing horizontal edge.

Consider $\det(I-\posgen \wt P)A_k$. By ~\Cref{lem:detP} we have $\det(I- \posgen \wt P) \doteq \prod_{i=1}^m (1-c_i) = \prod_{i=1}^m(g_i\alpha(\wt \eta_i))$. For each $i\ne k$, let us multiply the $\eta_i$--row of $A_k$ by the $g_i\alpha(\wt \eta_i)$ factor from $\det(I-\posgen \wt P)$. This produces a matrix $B'$ whose $\eta_i$--row is the linear combination $\sum_{x\in D_i \cup (E\setminus T)} b_x^i R_x$. Since $B'$ contains the row $R_x$ for each $x$ in $E\setminus T$, we may perform row operations to kill the terms $b_x^iR_x$ with $x\in E\setminus T$ in the $\eta_i$--row of $B'$. These operations yield a new matrix $B$ whose row for $x\in E\setminus T$ is $R_x$ and whose row for $\eta_i$ with $i\ne k$ is $\sum_{x\in D_i} b_x^i R_x$. Since these row operations do not change determinants, we further see that
\[\alpha(\wt\eta_k)\det(B) \doteq \det(I-\posgen \wt P) \det (A_k).\]

It thus remains to show that $\det(B) \doteq \det(I-\posgen \wt M)$. For this, let $\Lambda$ denote the set of horizontal edges of $T$, that is, $\Lambda = D_1 \cup \cdots \cup D_m$.
Impose a partial order on $\Lambda$ defined by $e \le e'$ if the path in the tree $T$ from $e'$ to $T_k$ passes through $e$. For each $i\ne k$, there is a unique tree path $\gamma$ from $T_i$ to $T_k$. This path $\gamma$ crosses exactly one edge of $D_i$, which we denote $y_i \in D_i$. Since the path from any other $x$ in $D_i$ to $T_k$ must follow $\gamma$ and thus pass through $y_i$, we see that $y_i \le x$ for all $x\in D_i$; hence $y_i$ is the minimal element of $D_i$. Conversely, each horizontal edge in $T$ is equal to the minimal edge $y_i \in D_i$ for a unique index $i\ne k$. Thus we have a bijection between $\Lambda$ and $\{1,\dots,m \}\setminus\{k\}$, so that $\Lambda = \{y_i \mid i \ne k\}$.

Say that a partition $\Lambda = \Lambda_1\sqcup \Lambda_2$ is $\le$--compatible if there do not exist $\lambda_j\in \Lambda_j$ such that $\lambda_2 \le \lambda_1$; that is, if for every pair $(\lambda_1,\lambda_2)\in \Lambda_1\times \Lambda_2$ either $\lambda_1 \le \lambda_2$ or else the pair is not ordered. Given such a partition let $B(\Lambda_1\sqcup \Lambda_2)$ be the matrix obtained from $B$ as follows: for each $i\ne k$ with $y_i \in \Lambda_2$, replace the row $\sum_{x\in D_i} b_x^i R_x$ in $B$ corresponding to $\eta_i$ with simply the row $R_{y_i}$ of $I-\posgen \wt M$ corresponding to $y_i\in D_i$. Clearly we have $B(\Lambda\sqcup \emptyset) = B$ and $B(\emptyset \sqcup \Lambda) = I-\posgen \wt M$. 

We claim by induction that $\det(B(\Lambda_1\sqcup \Lambda_2)) \doteq \det(B)$ for every $\le$--compatible partition $\Lambda = \Lambda_1\sqcup \Lambda_2$. This is immediate for the trivial partition $\Lambda\sqcup \emptyset$. Now let $\Lambda_1\sqcup \Lambda_2$ be any $\le$--compatible partition. If $\Lambda_2\ne \emptyset$, we may choose a $\le$--minimal element $\lambda\in \Lambda_2$. The adjusted partition $\Lambda_1' \sqcup \Lambda_2' = (\Lambda_1\cup\{\lambda\})\sqcup(\Lambda_2\setminus\{\lambda\})$ is then $\le$--compatible, and by induction we may assume $\det(B(\Lambda_1'\sqcup \Lambda_2')) \doteq \det(B)$. Let $i\ne k$ be the index so that $\lambda = y_i$. The matrices $B(\Lambda_1\sqcup \Lambda_2)$ and $B(\Lambda_1'\sqcup \Lambda_2')$ then only differ in the row corresponding to $\eta_i$, which in the latter is $\sum_{x\in D_i} b_x^i R_x$. Consider any edge $x\in D_i$ with $x\ne y_i$. Then by definition of $y_i$ we have $y_i \le x$. Therefore, the fact that $\Lambda_1\sqcup \Lambda_2$ is $\le$--compatible and $y_i\in \Lambda_2$ implies we cannot have $x\in \Lambda_1$. Hence $x\in \Lambda_2$ and, since $x\ne y_i = \lambda$, furthermore $x\in \Lambda_2'$. This means $B(\Lambda_1'\sqcup \Lambda_2')$ contains the row $R_x$ for each $x\in D_i$ with $x\ne y_i$. Thus we may perform a row operation on $B(\Lambda_1'\sqcup \Lambda_2')$ that uses the row $R_x$ to kill the $b_x^i R_x$ term in the $\eta_i$--row of $B(\Lambda_1'\sqcup \Lambda_2')$. Applying these operations for each $x\in D_i\setminus\{y_i\}$ produces a new matrix whose $\eta_i$--row is $b_{y_i}^i R_{y_i}$. Multiplying this row by the unit $(b_{y_i}^i)\inv$ then yields $B(\Lambda_1\sqcup \Lambda_2)$. Thus we have transformed $B(\Lambda_1'\sqcup \Lambda_2')$ into $B(\Lambda_1\sqcup \Lambda_2)$ via row operations that only effect determinants up to a unit. This proves the claim and, in particular, that $\det(I-\posgen \wt M) = \det(B)$ up to a unit in $\Z[H]$, as needed.
\end{proof}

\subsection{The determinant formula}
We will make use of the following lemma which can be found in work of Hironaka \cite[Lemma 8]{hironaka2011computing}. The argument also appears in the proof of \cite[Theorem 5.1]{McMullen--Alexander_polynomial}.
\begin{lemma}\label{lem:eko}
Suppose 
\[ \mathbb{Z}[H]^{n} \overset{A}{\longrightarrow} \mathbb{Z}[H]^{n+1} \overset{\alpha}{\longrightarrow} \mathbb{Z}[H], \]
is a sequence of $\mathbb{Z}[H]$--modules such that $\alpha \circ A = 0$. Let $\Delta_i$ be the determinant of the matrix obtained from $A$ after removing the $i$th row. Then
\begin{align}\label{eq:eko}
\alpha(e_i) \Delta_j = \pm \alpha(e_j) \Delta_i,
\end{align}
for all $1\le i,j \le n+1$. Here $\{e_i: 1 \le i \le n+1\}$ is the standard free basis of $\mathbb{Z}[H]^{n+1}$.
\end{lemma}

We are now ready to prove the determinant formula for the Alexander polynomial: 
\begin{introdet*}[Determinant formula for $\Delta_{\Gamma}$]
Let $f \colon G \to G$ be a graph map whose mapping torus $X$ has fundamental group $\Gamma$. 
Then
\[
\Delta_\Gamma \doteq \frac{\det(\dt I -\wt M)}{\det(\dt I-\wt P)} \cdot r,
\]
where $r=1\in \Z[H]$ if $\mathrm{rank}(H_1(X, \R)) \ge 2$ and $r = (\dt-1)$ 
if $\mathrm{rank}(H_1(X, \R)) =1$.
\end{introdet*}

\begin{proof}[Proof of \Cref{th:det_form}]
Recall that $|\mc E \setminus T| = |E| +1$ and, for $1 \le i \le |E|+1$,
let $\Delta_i$ be the minor of $A$ (from equation (\ref{eq:def_A})) obtained by removing the $i$th row.
Then the Alexander polynomial is by definition
\[
\Delta_{\Gamma} = \mathrm{gcd}\{\Delta_i : 1\le i\le |E|+1 \}.
\]
With $\alpha\colon \Z[H]^{\mathcal{E}\setminus T}\to \Z[H]$ from equation (\ref{eq:def_alpha}),
we note that $\alpha \circ A =0$ and so we apply \Cref{lem:eko} to conclude that 
\[
\alpha(\wt e_i) \Delta_j = \pm \alpha(\wt e_j) \Delta_i,
\]
where $\{\wt e_i \}$ are the base lifts of the edges of $\mc E \setminus T$. Fix some index, say $k$, representing a vertical edge $\eta_k$ of $\mc E \setminus T$ so that $\wt e_k = \wt \eta_k$ as in \Cref{lem:dets}. Then
\begin{align*}
\alpha(\wt \eta_k) \Delta_{\Gamma} &= \mathrm{gcd}\{\alpha(\wt \eta_k)\Delta_i : 1\le i\le |E|+1 \}\\
& = \mathrm{gcd}\{\alpha(\wt e_i): 1\le i\le |E|+1 \}  \cdot \Delta_k.
\end{align*}

But according to \Cref{lem:dets}, we know that,
\[
\Delta_k \det(I-\posgen \wt P) \doteq \alpha(\wt \eta_k)\det(I - \posgen \wt M).
\]
After rearranging, we conclude that
\[
\Delta_{\Gamma} \doteq \frac{\det(\dt I -  \wt M)}{\det(\dt I- \wt P)} \cdot \mathrm{gcd}\{\alpha(\wt e_i): 1\le i\le |E|+1 \},
\]
and so it suffices to show that $\mathrm{gcd}\{\alpha(\wt e_i): 1\le i\le |E|+1 \} = 1$ if $\mathrm{rank}(H) \ge 2$ and $\mathrm{gcd}\{\alpha(\wt e_i): 1\le i\le |E|+1 \} = \dt-1$ if $H = \langle \dt \rangle$. 

For this we recall the following: if $e_i$ is any edge of $\mathcal{E} \setminus T$ and $\wh e_i$ is its (nonbase) lift to $\wt X$ whose tail is contained in $\wt T'$, then $\alpha(\wh e_i) = g_i-1$, where $g_i$ is represented by the loop in $e \cup T$ crossing $e$ positively. Since $\wh e_i$ is a translate of the base lift $\wt e_i$ we have that, up to a unit, $\alpha(\wt e_i) = g_i-1$. Moreover, since the $\{g_i\}$ correspond to edges of $X^{(1)}$ outside of the maximal tree $T$, we note that $H = \langle g_1, \ldots, g_{|\mathcal{E} \setminus T|} \rangle$.

If $\mathrm{rank}(H) \ge 2$, then there are $g_i$ and $g_j$ that generate a rank $2$ subgroup of $H$ and hence $(g_i-1)$ and $(g_j-1)$ are relatively prime. In particular, $\mathrm{gcd}\{\alpha(\wt e_i)\} =\mathrm{gcd}\{(g_i-1)\} =1$, as required. 
If $\mathrm{rank}(H) = 1$, then $g_i = \posgen^{k_i}$ and since the $g_i$ generate $H$, in fact $\mathrm{gcd}\{k_i\} = 1$. Hence, we also have that $\mathrm{gcd}\{(\posgen^{k_i}-1)\} = (\posgen-1)$ and so $\mathrm{gcd}\{\alpha(\wt e_i)\} = (\dt-1)$. This completes the proof. 
\end{proof}

\subsection{Characteristic polynomials of monodromy}
The following corollary shows that specialization of the Alexander polynomial gives the characteristic polynomial of the monodromy's action on homology, up to replacing the indeterminate by its inverse. Similar statements appear in the literature where the characteristic polynomial considered is that of the deck transformation on the homology of the associated cyclic cover. See e.g.  Milnor \cite[Assertion 4]{milnor1968infinite}.

Recall that $f_* \colon H_1(G; \mathbb{Z}) \to H_1(G;\mathbb{Z})$ is the induced map on first homology of $G$.

\begin{corollary}\label{lem:spec_char_poly}
Let $X$ be the mapping torus of the graph map $f \colon G \to G$
and let $u$ be its dual class.
Then
\[
\det(t\inv I - f_*) \doteq (1-t)^p \cdot \Delta_{\Gamma}^u(t),
\]
where $p=1$ if $\mathrm{rank}(H_1(X))\ge 2$ and $p=0$ if $\mathrm{rank}(H_1(X)) =1$.
\end{corollary}

\begin{proof}
As before, let $M$ represent the action of $f$ on $1$--chains $\mathbb{Z}^E$ and  $P \colon \mathbb{Z}^V \to \mathbb{Z}^V$ denote the action of $f$ on $0$--chains. 

Let $\partial \colon \mathbb{Z}^E \to \mathbb{Z}^V$ be the usual boundary map on the oriented edges, and let $\alpha  \colon \mathbb{Z}^V \to \mathbb{Z}$ be the augmentation map defined by $\sum a_v v \mapsto \sum a_v$.
Then there is a diagram with exact rows:

\[
\begin{CD}
     0@>>>   H_1(G)  @>>>  \mathbb{Z}^E @>\partial >>  \mathbb{Z}^V @>\alpha>> \mathbb{Z} @>>> 0 \\
       @.      @VV f_* V     @VV M V                   @VV P V      @VV 1 V           @.     \\
       0@>>>   H_1(G)  @>>>  \mathbb{Z}^E @>\partial >>  \mathbb{Z}^V @>\epsilon>> \mathbb{Z} @>>> 0
\end{CD}
\]

From this, we have the equation of characteristic polynomials (in the variable $t\inv$), 
\[
(t\inv-1) \cdot \det(t\inv I - M) = \det(t\inv I - f_*)  \cdot \det(t\inv I- P).
\]

The proof is completed by applying \Cref{th:det_form} and noting that the $u$--specializations of $\det(\dt I-\wt M)$, $\det(\dt I-\wt P)$, and $\Delta_\Gamma$ are $\det(t\inv I - M)$, $\det(t\inv I- P)$, and $\Delta^u_\Gamma(\dt)$, respectively (since $u(\dt) = -1$). 
\qedhere
\end{proof}

\begin{remark}
\label{rmk:char_pol}
\Cref{lem:spec_char_poly} has the following immediate consequence: If graph maps $f \colon G \to G$ and $f' \colon G' \to G'$ produce equivalent mapping tori, in the sense that there is an isomorphism $\phi \colon \pi_1(X_f) \to \pi_1(X_{f'})$ for which the dual classes satisfy $u' \circ \phi = u$, then the characteristic polynomials of $f_*$ and $f'_*$ are the same, up to a unit in $\mathbb{Z}[t^{\pm1}]$. In particular, the homological stretch factors of $f$ and $f'$ agree.
\end{remark}

\section{The folded mapping torus, cross sections, and the McMullen polynomial}\label{sec:fmt}
Throughout this section, we fix a fully irreducible automorphism $\varphi$ of $\free$ and an irreducible train track map $f \colon G \to G$ representing $\varphi$. As in \S\ref{sec:mapping_tori}, this determines mapping torus $X = X_f$ equipped with a semiflow $\flow_t\colon X\to X$, coming from the local upward flow on $G\times[0,1]$, whose fundamental group $\Gamma$ is the free-by-cyclic group determined by $\varphi$. As before, we also set $H = H_1(X)/\text{torsion}$.
In this section, we briefly recall the main constructions and required results from \cite{DKL1,DKL2}.

\subsection{The folded mapping torus and its cross sections}
\label{sec:folded}
To construct cross sections representing different splittings of $\Gamma$, Dowdall--Kapovich--Leininger \cite{DKL1} constructed a modified mapping torus equipped with its own semiflow.  Given a fixed factorization of the graph map $f \colon G \to G$ into a sequence of Stallings folds, the authors produce a $2$-complex $X^\dagger = X^\dagger_f$ which they call the \emph{folded mapping torus}. 
Just as for the mapping torus, $X^\dagger$ comes equipped with a semiflow $\flow^\dagger_t\colon X^\dagger\to X^\dagger$ and the two semiflows are related by a natural flow-equivariant quotient map $q\colon X\to X^\dagger$ with the property that the set of vertex leaves of $X$ is mapped bijectively 
to the set of vertex leaves of $X^\dagger$. 
Here, a \emph{vertex leaf} is a connected component of the set of points whose forward orbit under the semiflow meets a vertex.
The map $q$ is a homotopy equivalence and we use it to once and for all identify the fundamental groups and homology groups of the $2$-complexes.

\smallskip
Following \cite[\S5.1]{DKL2}, a \emph{cross section} of $X^\dagger$ is a finite embedded graph $\Theta$ that is transverse to the flow such that every flowline hits $\Theta$ infinitely often in the sense that $\{s\in \R_{\ge 0} \mid \flow^\dagger_s(x)\in \Theta\}$ is unbounded for every $x\in X^\dagger$.  In this case $\Theta$ has continuous \emph{first return map} $f_\Theta\colon \Theta\to \Theta$ that sends $x\in \Theta$ to the next point at which the flowline $\flow^\dagger_s(x)$ intersects $\Theta$.

\begin{remark}
\label{rem:non-injective-cross-sections}
We caution that in general a cross section $\Theta$ need not be $\pi_1$--injective and the first return map $f_\Theta$ need not be a homotopy equivalence.
\end{remark}

For our purposes, we will only consider cross sections that are  \emph{compatible} with $\flow^\dagger$, which means the intersection of $\Theta$ with the $1$--skeleton of $X^\dagger$ is a finite set contained in the vertex leaves. 
In this case, $\Theta$ may be equipped with a finite \emph{standard graph structure} \cite[Definition 7.3]{DKL2} in which every edge lies in a $2$--cell of $X^\dagger$ and has endpoints in vertex leaves. Then, the first return map $f_\Theta$ is a train track map 
with irreducible transition matrix $A_\Theta$ and spectral radius $\lambda(f_\Theta) > 1$ \cite[Proposition 7.7]{DKL2}. According to \cite[Corollary 7.9]{DKL1}, the train track map $f_\Theta$ satisfies the additional property that the Whitehead graphs at each of its vertices are connected (note that the proof of \cite[Corollary 7.9]{DKL1} does not use the standing assumption of \cite[Convention 7.6]{DKL1} that $f$ is atoroidal).

For any cross section $\Theta$ of $X^\dagger$, the semiflow $\flow^\dagger$ can be reparameterized so that the return time to $\Theta$ is constant and equal to $1$ (see \cite[Definition 5.1]{DKL2}).
 Then, just as above, there is a natural flow preserving quotient map $q_\Theta$ from the mapping torus $X_{f_\Theta}$ of the first return map $f_\Theta$ onto $X^\dagger$ determined by
\begin{align} \label{eq:flowhmtpy}
\Theta \times [0,1] &\to X^\dagger \\
(x,t) &\mapsto \flow^\dagger_t(x). \nonumber
\end{align}
Since the induced flow preserving map $q_\Theta \colon X_{f_\Theta} \to X^\dagger$ is a $\pi_1$--isomorphism, it is a homotopy equivalence. 
Compatibility of the cross section $\Theta$ implies that the set of vertex leaves of $X_{f_\Theta}$ is mapped bijectively to the set of vertex leaves of  $X^\dagger$, just as for the original map $q$ defined above. In what follows, all cross sections are assumed to be compatible.

For any cross section $\Theta$, the homotopy equivalence $q_\Theta \colon X_{f_\Theta} \to X^\dagger$ identifies the fundamental groups and homology groups of the $2$-complexes. Hence, any cross section determines a \emph{dual cohomology class} 
\[
[\Theta] \in H^1(X_{f_\Theta};\Z) \cong H^1(X^\dagger;\Z)  
= \mathrm{Hom}(\Gamma,\Z) = \mathrm{Hom}(H, \Z).
\]

In fact, the classes dual to cross sections are precisely the primitive integral points of an open, rational cone $\mc C_{\mf m}$ called the \emph{cone of cross sections} (see \Cref{thm:DKL-cone-theorem}). To explain, we must first recall the McMullen polynomial. 

\subsection{The McMullen polynomial and the cone of cross sections}
\label{sec:Mccone}
We next describe a polynomial invariant of the dynamical system $(X^\dagger, \flow^\dagger)$ that was constructed in \cite[\S4.1]{DKL1}.
Continue to let $H = H_1(X^\dagger)/\mathrm{torsion}$ and let $\wt X^\dagger\to X^\dagger$ be the free abelian cover with deck group $H$ and with lifted semiflow $\wt{\flow}^\dagger_t\colon \wt X^\dagger\to \wt X^\dagger$.

Let us define a \emph{transversal} of $\wt \flow^\dagger$ to be an arc $\tau$ contained in a $2$--cell of $\wt X^\dagger$ that is transverse to the flowlines and has its endpoints in vertex leaves. The action of $H$ on $\wt X^\dagger$ by deck transformations induces an action on transversals and makes the free $\Z$--module $\mathbb{F}(\wt \flow^\dagger)$ on the set of transversals into a $\Z[H]$--module.

The \emph{module of transversals} $T(\wt \flow^\dagger)$ is the quotient of $\mathbb{F}(\flow^\dagger)$ by the submodule generated by subdivision relations $\tau - \tau_1 - \tau_2$ for all transversals where $\tau = \tau_1\cup \tau_2$ and $\tau_1\cap \tau_2$ is a point in a vertex leave, and flow relations $\tau - \tau'$ whenever a transversal $\tau$ flows homeomorphically on to $\tau'$. 
Proposition 4.3 of \cite{DKL2} proves that $T(\wt \flow^\dagger)$ is a finitely presented $\Z[H]$--module.

\begin{definition}\label{def:mcdef}
The \emph{McMullen polynomial} $\mf m\in \Z[H]$ of the system $(X^\dagger,\flow^\dagger)$ is the $0^\mathrm{th}$ Fitting invariant of $T(\wt \flow^\dagger)$; that is, gcd of the $0^\mathrm{th}$ Fitting ideal of the $\Z[H]$--module $T(\wt \flow^\dagger)$.
\end{definition}

\begin{remark}
\label{rem:different_action_from_DKL} 
We use the natural action of $H$ on transversals given by $h\cdot \tau = h(\tau)$ for $h\colon \wt X^\dagger\to \wt X^\dagger$ a deck transformation. This differs from the convention used in \cite{DKL2}, where $H$ was declared to act on transversals by taking preimages.
 As a result, our McMullen polynomial $\mf m$ differs from the polynomial $\mf m_{\mathrm{DKL}}$ introduced in \cite{DKL2} by the ring automorphism $\mathrm{inv}\colon\Z[H]\to \Z[H]$ induced by the automorphism $h\mapsto h^{-1}$ of the abelian group $H$; that is,  $\mf m = \mathrm{inv}(\mf m_{\mathrm{DKL}})$.  
The convention here is chosen to more closely parallel the construction of the Alexander polynomial, whereas the convention in \cite{DKL2} was used to parallel features of the Teichm\"uller polynomial in \cite{McMullen--poly_invariants}.
\end{remark}

While \Cref{def:mcdef} is rather opaque, \cite[Theorem D]{DKL1} proves that $\mf m$ may be explicitly calculated as a determinant with respect to any compatible cross section $\Theta\subset X^\dagger$. To explain this, let $\wt\Theta\subset \wt X^\dagger$ be the full preimage and let $\wt\Theta_0$ be a designated connected component. Choose also a lift $\wt f_\Theta\colon \wt \Theta_0\to \wt \Theta_0$ of $f_\Theta$ to $\wt \Theta_0$.
By using the flow preserving homotopy equivalence $q_\Theta \colon X_{f_\Theta} \to X^\dagger$ from \Cref{eq:flowhmtpy}, we see that
exactly as in \S\ref{sec:abelian_cover}, these choices determine a splitting $H\cong K_\Theta\oplus \langle \dt_\Theta\rangle$ of the deck group $H$ 
as well as an identification $C_1(\wt \Theta_0)\cong \Z[K_\Theta]^{E\Theta}$ with respect to which the transition matrix of $\wt f_\Theta$ is expressed as an $E\Theta\times E\Theta$ matrix $\wt A_\Theta$ with entries in $\Z[K_\Theta]$ that specializes to $A_\Theta$ under the augmentation map $\Z[K]\to \Z$.
With this notation, we have the following calculation of $\mf m$:

\begin{theorem}[Determinant formula {\cite[Theorem D]{DKL2}}]
\label{thm:mcpoly_det_formula}
Let $\Theta$ be any connected, compatible cross section of $(X^\dagger, \flow^\dagger)$ with associated splitting $H = K_\Theta \oplus \langle \dt_\Theta\rangle$ as above. Then the McMullen polynomial of $(X^\dagger, \flow^\dagger)$ is
\[\mf m \doteq \det(\dt_\Theta I - \wt A_\Theta).\]
\end{theorem}

\begin{remark}  \label{rmk:mcMpoly_under_inv}
In the coordinates $\Z[H]\cong \Z[{\bf t}^{\pm}, x^\pm]$ used in \cite{DKL2}, the determinant formula is written $\mf m_{\mathrm{DKL}} \doteq \det(xI - \wt A(\mathbf{t}))$, where $x$ is the inverse of our element $\dt_\Theta$ and $\wt A(\bf t)$ describes $\wt f_\Theta$ acting on $C_1(\wt \Theta_0)$ with respect to the module structure in which deck transformations act by taking preimages. Hence $\mathrm{inv}\colon \Z[H]\to \Z[H]$ sends $x$ to $\dt_\Theta$ and $\wt A(\mathbf{t})$ to $\wt A_\Theta$, so that we indeed have $\mf m = \mathrm{inv}(\mf m_{\mathrm{DKL}}) \doteq \mathrm{inv}(\det(xI - A(\mathbf{t}))) = \det(\dt_\Theta I - \wt A_\Theta)$.
\end{remark}
Notice that the dual cohomology class $[\Theta]\colon H\to \Z$ sends $\dt_\Theta$ to $-1$ and $K_\Theta$ to $0$; hence  the induced ring map $\Z[H]\to \Z[\Z]\cong \Z[t^\pm]$ sends $\dt_\Theta$ to $t^{-1}$ and $\wt A_\Theta$ to the $E\Theta\times E\Theta$ integer matrix $A_\Theta$. Applying \Cref{thm:mcpoly_det_formula}, we conclude that the specialization of $\mf m$ at $[\Theta]$ (see \S\ref{sec:invariants}) is simply the reciprocal characteristic polynomial of the transition matrix $A_\Theta$ of $f_\Theta$:
\begin{equation}
\label{eqn:specialized_det_formula} 
\mf m^{[\Theta]}(t) \doteq \det(t^{-1} I - A_\Theta).
\end{equation}

To summarize: the first return map $f_\Theta$ is a train track map with irreducible transition matrix $A_\Theta$ whose characteristic polynomial is $\mf m^{[\Theta]}(t^{-1})$. Hence the geometric stretch factor $\lambda(A_\Theta)>1$ of $f_\Theta$ is the reciprocal of the smallest root of $\mf m^{[\Theta]}(t)$.

The construction picks out a distinguished term of the McMullen polynomial; namely the $z^{\abs{E}}$ term of $\det(zI - \wt A_{G^\dagger})$ in the coordinates $H = K\oplus \langle z \rangle$ adapted to the base cross section $G^\dagger$. Writing 
$\mf m = a_0 h_0 + \dots + a_n h_n\in \Z[H]$ with $a_0 h_0$ this distinguished term, the polynomial determines a corresponding \emph{cone of cross sections}
\[\mathcal{C}_{\mf m} = \{u\in H^1(X^\dagger;\R) \mid u(h_0) < u(h_i)\text{ for each }i=1,\dots, n\}.\]
Notice that this is an open polyhedral cone with finitely many rationally defined sides.
The significance of $\mathcal{C}_{\mf m}$ (and a justification of its name) is captured by the following main result from \cite{DKL2}:

\begin{theorem}[{\cite[Meta-Theorem I]{DKL2}}]
\label{thm:DKL-cone-theorem}
An integral cohomology class $u\in H^1(X^\dagger;\Z)$ is dual to a cross section if and only if $u\in \mathcal{C}_{\mf m}$; this cross section may be chosen to be compatible with $\flow^\dagger$ and is connected if and only if $u$ is primitive.

Further,  $\mathcal{C}_{\mf m} \subset H^1(X^\dagger;\R)$ is equal to the set of cohomology classes  that are positive on every 
closed orbit of the semiflow $\flow^\dagger$, and also to
 the component of the BNS-invariant $\BNS(\Gamma)$ containing the dual class $\pi_1(G)_{f_*}\to \Z$ of the original HNN splitting of $\Gamma$.
\end{theorem}

\subsection{McMullen polynomials via Perron polynomials}
In this section, we present an alternative characterization of the McMullen polynomial (\Cref{prop:alt_def}) that 
will be essential in understanding the relation between the McMullen polynomial and the Alexander polynomial in the negatively orientable setting. 

Let $D$ be the directed graph with a vertex for each edge of $G$ and a directed edge from $e$ to $e'$ for each time $f(e)$ crosses $e'$ with either orientation. That is, $D$ is the directed graph whose adjacency matrix is $A$, and so in particular $D$ is strongly connected. Recalling the cell structure of $X_f$ discussed in \S\ref{sec:mapping_tori}, there is an embedding $i \colon D \to X_f$ obtained by mapping each vertex of $D$ to the midpoint of the corresponding horizontal edge of $X_f$ and mapping the edges out of the $e$-vertex into the $2$--cell $\sigma_e$ in the obvious way. See \Cref{fig:D}. The map $i \colon D \to X_f$ is an embedding and we often identify $D$ with its image.

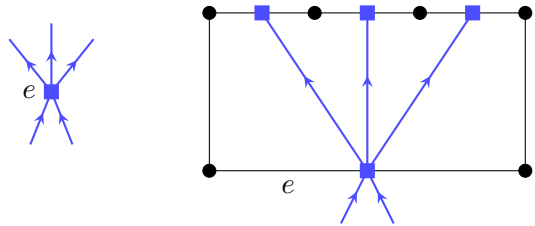
\begin{figure}[ht]
\begin{tikzpicture}[scale=.7]
\draw (0,0) -- (6,0) -- (6,3) -- (0,3) -- (0,0);
\path (1.5,0) node[below] {$e$};
\path (-3.1,1.5) node[left] {$e$};
\draw [black] plot [only marks, mark size=3.5, mark=*] 
 coordinates {
(0,0) (6,0) 
(0,3) (2,3) (4,3) (6,3)
};
\begin{scope}[blue!70!white,thick,decoration={
    markings,
    mark=at position 0.6 with {\arrow{stealth}}}
  ] 
\draw[postaction={decorate}] (3,0) -- (1,3); 
\draw[postaction={decorate}] (3,0) -- (3,3);
\draw[postaction={decorate}] (3,0) -- (5,3);
\draw[postaction={decorate}] (2.5,-1) -- (3,0);
\draw[postaction={decorate}] (3.5,-1) -- (3,0);

\draw[postaction={decorate}] (-3,1.5) -- (-3.8, 2.5);
\draw[postaction={decorate}] (-3,1.5) -- (-3, 2.8);
\draw[postaction={decorate}] (-3,1.5) -- (-2.2, 2.5);
\draw[postaction={decorate}] (-2.6 ,.5 ) -- (-3,1.5);
\draw[postaction={decorate}] (-3.4 ,.5 ) -- (-3,1.5);
\draw plot [only marks, mark size=3.5, mark=square*] 
 coordinates {
(-3, 1.5) (3,0)
(1,3) (3,3) (5,3)
};
 \end{scope}
\end{tikzpicture}
    \caption{The image of $D$ inside the $2$-cell $\sigma_e$ whose bottom edge is $e$. The original vertices of $X_f$ are black dots and the midpoint vertices (i.e.~the images of vertices of $D$) are blue squares.}
\label{fig:D}
\end{figure}

\begin{lemma} \label{lem:surject}
The embedding $i \colon D \to X_f$ is $\pi_1$--surjective.
\end{lemma}

\begin{proof}
Identify $D$ with its image under $i$.
Begin by subdividing the horizontal edges of $X_f$ at the vertices of $D$ and call these \emph{midpoint vertices}. It suffices to show that if $p$ is an edge path in $X_f^{(1)}$ that meets the midpoint vertices exactly at its endpoints then it is homotopic rel vertices into $D$. 

First assume that $p$ is horizontal, i.e.~contained in $G \subset X_f$. Then $p$ contains exactly one vertex $v$ of $G$ and two half edges with midpoints $m_a$ and $m_b$; we denote these (directed) half edges by $m_av$ and $vm_b$, respectively, so that $p = m_av *vm_b$. Note that if $m$ is the midpoint vertex of an edge $e$ with both its endpoints at $v$, then the notation $mv$ is ambiguous, but in what follows the proper meaning will be clear from context.

 If the path $p$ lies in the top of a $2$-cell, then the claim is clear since $p$ is homotopic rel endpoints to a (nondirected) path in $D$ of length $2$. See \Cref{fig:D}. Otherwise, we consider the vertices $m_a$ and $m_b$ in the Whitehead graph $\mathrm{Wh}_G(v)$. Since $f$ represents a fully irreducible automorphism, $\mathrm{Wh}_G(v)$ is connected and so there is a path $m_a, m_1,\ldots, m_k ,m_b$ in $\mathrm{Wh}_G(v)$. By definition of the Whitehead graph, for each $i$, the corresponding path $m_ivm_{i+1}$ in $G$ is contained in the image of an edge under $f^n$ for some $n\ge 1$. We claim that these paths are each homotopic into $D$. If $n=1$, then $m_ivm_{i+1}$
lies at the top of a $2$-cell, so by the previous case, it is homotopic into $D$. Otherwise, the turn of $G$ at $v$ associated to $m_ivm_{i+1}$ is the image of the turn at a vertex $w$ associated to a path $m_i'wm_{i+1}'$ in $G$, and $m_i'wm_{i+1}'$ is
traversed by the image of an edge under $f^{n-1}$. By induction, $m_i'wm_{i+1}'$ is homotopic into $D$ and if we denote by $c_i$ and $c_{i+1}$ the $2$--cells whose bottom edges contain the midpoint vertices $m_i'$ and $m_{i+1}'$, respectively, we see that $c_i$ and $c_{i+1}$ share the vertical edge from $w$ to $v$. Moreover, $c_j$ contains a directed edge $d_j$ of $D$ from $m_j'$ to $m_j$ for each $j \in \{i, i+1\}$. Hence, $m_ivm_{i+1}$ is homotopic in $c_i \bigcup c_{i+1}$ to the path $\overline{d}_{i}*m_i'wm_{i+1}'*d_{i+1}$ and hence homotopic into $D$.
 As $p = m_avm_b$ is homotopic to the concatenation $m_a v m_1*m_1vm_2* \ldots * m_k v m_b$, this shows that every horizontal $p$ is homotopic into $D$.

Next, suppose that $p$ contains a single vertical edge $e$. If this path lies in the boundary of a single $2$-cell, then it is again clear from \Cref{fig:D} that $p$ is homotopic into $D$. Otherwise, $p$ has the form $m_a v *e*w m_b$, where $e$ is a vertical edge from $v$ to $w$. Let $\sigma$ be the $2$-cell containing the vertical edge $e$ and $m_a$ along its bottom edge. Let $m_c$ be the midpoint at the top of $\sigma$ nearest $e$. Then $p$ is homotopic to $m_a v *e*wm_c*m_cw m_b$, which we can handle by the previous cases. 
Finally, if $p$ contains multiple vertical edges, then it homotopic to a concatenation of paths of this form $m_a v \ast e \ast w m_b$ and we finish as above. 
\end{proof}

For the embedding $i \colon D \to X_f$, we denote by $i_*$ both the induced group homomorphism
$H_1(D) \to H$ and ring homomorphism $ \mathbb{Z}[H_1(D)] \to \mathbb{Z}[H]$.

\medskip
Now let $P_D$ be the Perron polynomial of $D$, defined as $P_D := \det(I -A)$ where $A$ is the symbolic transition matrix of $D$ given by
\[
A_{xy} = \sum_{\partial e = y-x} e,
\]
with coefficients in $\mathbb{Z}[C_1(D)]$. We refer the reader to \cite{mcmullen2015entropy} for additional details and recall that according to \cite[Section 3]{mcmullen2015entropy} or \cite[Theorem 2.14]{algom2015digraphs},
\begin{align} \label{eq:clique}
P_D = 1 + \sum_{\sigma} (-1)^{\vert \sigma \vert}\sigma \in \mathbb{Z}[H_1(D)], 
\end{align}
where the sum varies over nonempty oriented multi-cycles $\sigma$ in $D$ and $\vert \sigma |$ is the number of components of the multi-cycle.

The next proposition provides an explicit connection between polynomial invariants defined in \cite{DKL2} and \cite{algom2015digraphs}.
The first claim is essentially proven by \cite[Theorem 12.10]{DKL2}, but we give a direct proof using a result from \cite{landry2020polynomial}.

\begin{proposition} \label{prop:alt_def}
With $i \colon D \to X_f$ as above, 
 $\mf m \doteq i_*(P_D)$.
Moreover, $\mf m' = i_*(P_D)$ is the unique normalization of $\mf m$ such that $\supp(\mf m')$ contains $1$ and for some (any) $u \in \mc C_{\mf m}$, $u$ is positive on $\supp(\mf m') \setminus\{1\}$.
\end{proposition}

\begin{proof}
First, \Cref{lem:surject} implies that the homomorphism $i_* \colon H_1(D) \to H$ is surjective. By \cite[Proposition 4.2]{landry2020polynomial}, $ i_*(P_D)$ is equal to the polynomial $P_{D,i_*}$ defined as follows: consider the preimage $\wt D$ of $D$ in the universal free abelian cover $\wt X_f$ with deck group $H$. 
The polynomial $P_{D,i_*} \in \mathbb{Z}[H]$ is defined as $\det L$, where $L$ is the endomorphism of the free $\mathbb{Z}[H]$-module generated by $H$--orbits of vertices of $\wt D$ given by $L(v) = v - (v_1 + \ldots + v_l)$. Here, the vertices $v_1, \ldots, v_l$ are exactly the endpoints (with multiplicity) of the directed edges of $\wt D$ out of $v$. Since $D$ is the graph associated to the adjacency matrix $A$ of $f$, we observe that $L = I-\ol \dt \wt A$.
From \Cref{thm:mcpoly_det_formula}, we conclude the required equality:
\[
i_*(P_D) = P_{D,i_*} = \det(I-\ol \dt \wt A) = {\ol \dt}^k \mf m.
\]

For the moreover statement, we first claim that $\mf m' = i_*(P_D)$ has the required properties. This follows from the fact that $P_D$ has the properties that $1 \in \supp(P_D)$ by \cref{eq:clique} and each $a \in \supp(P_D) \setminus \{1\}$ has $u(i(a))>0$ for any $u \in \mc C_{\mf m}$ by \cite[Proposition 6.3]{DKL2}. Therefore, $1 \in \supp(i_*(P_D))$ and the claim is established.

Next suppose that $\mf m''$ is another such normalization. Since $\mf m'$ and $\mf m''$ both contain $1$ in their support and differ by a unit, $\mf m' = h \; \mf m''$ implies that $h \in \supp( \mf m')$ and $h^{-1} \in \supp{(\mf m'')}$.
This must mean that $h=1$ since any $u \in \mc C_{\mf m}$ has to be positive on 
$h$ and $h^{-1}$.
\end{proof}

\begin{proposition} \label{prop:mcgenerate}
Let $\mf m'$ be the normalization of $\mf m$ from \Cref{prop:alt_def}. Then support of $\mf m'$ generates $H$ as an abelian group.
\end{proposition}

\begin{proof}
From \Cref{eq:clique}, we observe that the support of $P_D$  contains (the homology classes of) all of the simple directed cycles of $D$ and these generate $H_1(D)$ as an abelian group (see e.g. \cite[Lemma 5.8]{landry2020polynomial}) . It follows from \Cref{lem:surject} that the $i$--images of these directed cycles generate $H$ and so 
$i_*(\supp(P_D))$ generates $H$. 

By \Cref{prop:alt_def}, $\supp(\mf m')$ is equal to
$\supp(i_*(P_D))$, which may be properly contained in 
$i_*(\supp(P_D))$. Nevertheless, we now show that the span of $\supp(i_*(P_D))$ is equal to the the span of $i_*(\supp(P_D))$. We will consider the class $u \in H^1(X_f)$ dual to the projection $X_f \to S^1$ as a `norm'; it has the property that for each directed cycle $y$ of $D$, $u(i(y)) >0$.

Say $x \in \supp(P_D)$ \emph{survives} if $i_*(x) \in \supp(\mf m')$. Then we are left to show that the image of every simple cycle is a combination of the images of surviving simple cycles. 

Suppose to the contrary that there is a simple directed cycle $x \in \supp(P_D)$ such that $i(x)$ is not in the span of the images of surviving simple directed cycles. Among all such cycles, let $x$ be one such that $u(x)$ is minimal. Evidently $x$ does not survive.   Since $x$ is simple and not surviving, there is a $y \in \supp(P_D)$ represented by a multi-cycle $\sigma$ with an even number of components $\abs{\sigma}$ such that $ i_*(x) = i_*(y)$. Hence, $y$
can be written as a disjoint union $d_1 \cup \ldots \cup d_k$ of simple directed cycles in $D$ with $k\ge 2$. Then $i_*(x) = \prod i_*(d_j)$ and so $u(i_*(d_j)) < u(i_*(x))$ for each $j$ (recalling  that $u(i_*(d_j))>0$). But since $x$ is $u$--minimal, we see that each $i_*(d_j)$ must be the span of surviving simple cycles. As $i_*(x)$ is a combination of the $i_*(d_j)$, we have a contradiction and the proof is complete. 
\end{proof}

\subsection{The vertex polynomial}
\label{sec:vertex_poly}
For the folded mapping torus $X^\dagger$, we define the \emph{vertex cycles} of $X^\dagger$ to be 
 the closed orbits of $\flow^\dagger$ that pass through vertices of $X^\dagger$. The vertex cycles form a finite collection of disjoint, embedded closed orbits.

Let $c_1, \ldots, c_m$ be the vertex cycles of $X$. The \emph{vertex polynomial} $\mathfrak{p}$ of $X^\dagger$ is defined to be the element
\[
\mathfrak{p} = \prod_{i=1}^m \left(1-c_i \right) \in \mathbb{Z}[H],
\]
where we have identified $c_i$ with its class in $H$. 

The vertex polynomial can also be explicitly computed as a determinant similar to the McMullen polynomial. For this, we recall the setup before \Cref{thm:mcpoly_det_formula}: 
Fix a (compatible) cross section $\Theta$ of the semiflow $\flow^\dagger$ on $X^\dagger$. 
Let $\wt\Theta\subset \wt X^\dagger$ be the full preimage and let $\wt\Theta_0$ be a designated connected component. Choose also a lift $\wt f_\Theta\colon \wt \Theta_0\to \wt \Theta_0$ of $f_\Theta$ to $\wt \Theta_0$. As before this lift determines a splitting $H\cong K_\Theta\oplus \langle \dt_\Theta\rangle$ of the deck group $H$.
Letting $V\Theta$ denote the set of vertices of $\Theta$, choosing base lifts of these to $\wt \Theta_0$ again gives an identification $C_0(\wt \Theta_0)\cong \Z[K_\Theta]^{V\Theta}$ with respect to which $\wt f_\Theta$'s action on vertices is given as an $V\Theta\times V\Theta$ matrix $\wt P_\Theta$ with entries in $\Z[K_\Theta]$.

\begin{lemma} \label{lem:vertex_poly_formula}
With notation as above, $\mathfrak{p} \doteq \det(\dt_\Theta I-\wt P_\Theta)$.
\end{lemma}

\begin{proof}
Let $X_{f_\Theta}$ be the folded mapping torus of the return map $f_\Theta$ and let  $q_\Theta \colon X_{f_\Theta} \to X^\dagger$  be flow preserving homotopy equivalence from \eqref{eq:flowhmtpy}. Since $\Theta$ is compatible, $q_\Theta$ induces a bijection between the vertex cycles of the flows. Hence $\mf p = \prod_{i=1}^m(1- [\mc V_i'])$, where $\{\mc V_1',\ldots, \mc V_m'\}$ is the (disjoint) collection of closed orbits of $X_{f_\Theta}$ through vertices. From \Cref{lem:detP}, we then see that $\mf p \doteq \det(\dt I - \wt P)$. Since $q_\Theta$ identifies the splitting $H = K \oplus \langle \dt \rangle$ (from \S\ref{sec:abelian_cover}) with the splitting $H =  K_\Theta\oplus \langle \dt_\Theta\rangle$, this completes the proof.
\end{proof}

\section{Relating the polynomials and applications}
\label{sec:relating}
In this section we prove our main theorems relating the McMullen and Alexander polynomials.
This is done first in the positively orientable case in \S\ref{sec:orient_polys}, then in the negatively orientable case in \S\ref{sec:anti-orientable}, and finally in the general case in \S\ref{sec:gen_orient_polys}.

For the entirety of this section we fix a fully irreducible automorphism $\varphi\colon \free\to \free$ represented by an irreducible train track $f \colon G \to G$. Let $X^\dagger$ be an associated folding mapping torus with fundamental group $\Gamma = \pi_1(X^\dagger)$, let $\mf m, \Delta_\Gamma\in \Z[H]$ be its McMullen and Alexander polynomials, and $\mc C_{\mf m} \subset H^1(X^\dagger;\R)$ its cone of sections.

Since $\mc C_{\mf m}$ is a component of $\BNS(\Gamma)$ (\Cref{thm:DKL-cone-theorem}), we know that every integral class $u\in  \mc C_{\mf m}$ is dual to a splitting of $\Gamma$ as a generalized HNN extension $B*_\phi$ over a finitely generated free group $B$ (see \S\ref{sec:intro_BNS}). In this case, we say that $\phi$ is a (not necessarily unique) \emph{monodromy} associated to $u$. \Cref{lem:stretchdef} guarantees that the monodromy may be chosen to be injective and that the stretch factors  $\rho_\phi$,  $\lambda_\phi$ and characteristic polynomial of $\phi_*$ acting on $H_1(B;\Z)$ (up to a monomial factor) depend only on the class $u$ and not on the chosen monodromy:

\begin{definition}
Given a primitive integral class $u\colon \Gamma\to \Z$ in $\mc C_{\mf m}$, we define its homological  and geometric  \define{stretch factors}, respectively denoted $\rho(u)$ and $\lambda(u)$, to be the stretch factors $\rho_\phi$ and $\lambda_\phi$ of any monodromy $\phi$ associated to $u$. We call $u$ \define{orientable} if $\lambda(u) = \rho(u)$; such classes are moreover either \define{positively} or \define{negatively orientable}, respectively, 
if $\lambda(u)$ or $-\lambda(u)$ is a root of the characteristic polynomial of an associated monodromy $\phi$.
\end{definition}

The next proposition justifies this terminology and demonstrates that the various ways a cohomology class can be orientable all agree.

\begin{proposition} \label{prop:all_orient}
The following are equivalent for  primitive integral classes $u \in \mc C_{\mf m}$:
\begin{enumerate}
\item $u$ is orientable, i.e.~$\lambda(u) = \rho(u)$,
\item any injective monodromy $\varphi_u$ associated to $u$ has an orientable graph map representative,
\item for any compatible cross section $\Theta_u$ dual to $u$, the associated first return map $f_u \colon \Theta_u \to \Theta_u$ is orientable.
\end{enumerate}
Moreover, positive (negative) orientability in one case implies the same in all cases.
\end{proposition}

\begin{proof} 
If $\varphi$ is atoroidal, then $\Gamma$ is hyperbolic and $\varphi_u$ is also fully irreducible by \cite[Theorem 1.2]{DKL-endos} or \cite[Theorem 4.5]{mutanguha--invariance_of_iwip}.
Otherwise, $\varphi$ is toroidal and represented by a pseudo-Anosov homeomorphism on a once-punctured surface \cite[Theorem 4.1]{BestvinaHandel-TTs}, in which case $\mc C_{\mf m}$ is the cone over a fibered face of the Thurston norm ball and all associated monodromies, including $\varphi_u$, represent pseudo-Anosov homeomorphisms \cite{thurston1986norm}
 (see also \cite[Theorem 3.4]{mutanguha--invariance_of_iwip}). 
In either case, \Cref{lem:get_prim} implies that $\varphi_u$ admits a primitive train track representative $F_u$.

The first return map $f_u \colon \Theta_u \to \Theta_u$ is an irreducible train track map with connected Whitehead graphs. Hence, it is also primitive by \cite[Proposition 2.6]{mutanguha2020irreducible}.  By definition, the homomorphism on $\pi_1$ induced by $f_u$ is a monodromy associated to $u$, although it may not be injective. Thus 
\Cref{lem:stretchdef} implies $\lambda_{f_u} = \lambda(u) = \lambda_{F_u}$ and $\rho_{f_u} = \rho(u) = \rho_{F_u}$. The proposition now follows immediately from \Cref{th:stretch_or}.
\end{proof}

\begin{remark}
\label{rem:iso_stretch_implies_orient_endo}
As indicated in the proof above, if $\varphi$ is atoroidal, then any injective monodromy $\varphi_u$ associated to a class $u\in \mc C_{\mf m}$ is also a fully irreducible free group endomorphism. Hence in this case \Cref{thm:orientable-auto} implies the conditions in \Cref{prop:all_orient} are also equivalent to the monodromy $\varphi_u$ itself being (pos/neg) orientable.
\end{remark}

\subsection{Positively orientable case} 
\label{sec:orient_polys}

In this subsection, we relate the two polynomial invariants of a positively orientable fully irreducible automorphism. 

\begin{theorem} \label{th:polynomials}
Suppose the fully irreducible automorphism $\varphi$ is positively orientable. 
If $\mathrm{rank}(H_1(\Gamma))\ge 2$, then the McMullen and Alexander polynomials are related by
\[
\mf m \doteq \Delta_\Gamma \cdot \mf p,
\]
Otherwise, the equation is  $\mf m\cdot (1-z) \doteq \Delta_\Gamma \cdot \mf p$, where $\dt$ generates $H_1(\Gamma) / \mathrm{torsion}$.
\end{theorem}

\begin{proof}
By \Cref{fact:matrix}, for a positively orientable graph map the induced map on cellular $1$-chains and the transition matrix are equal. Lifting to a map $\wt f \colon \wt G_0 \to \wt G_0$, the matrices used in the definition of $\mf m$ and in \Cref{th:det_form} are equal, i.e.~$\wt A = \wt M$. This, together with the 
characterizations of $\mf m$ in \Cref{thm:mcpoly_det_formula} and $\mf p$  in  \Cref{lem:vertex_poly_formula}, completes the proof. 
\end{proof}

\begin{theorem} \label{th:orientable_cone}
If some primitive integral class in $\mc C_{\mf m}$ is positively orientable then so is every primitive integral class in $\mc C_{\mf m}$.
\end{theorem}

\begin{proof}
Let $u$ be a primitive integral class in $\mc C_{\mf m}$ represented by a cross section $\Theta = \Theta_u$ compatible with the induced semiflow (\Cref{thm:DKL-cone-theorem}) and let $f_\Theta \colon \Theta \to \Theta$ denote its first return map. 
As in \Cref{eq:flowhmtpy}, there is a flow preserving homotopy equivalence 
\[
q_{\Theta} \colon X_{f_{\Theta}} \to X^\dagger
\]
sending vertex leaves to vertex leaves, which we use to identify $H_1(X_{f_\Theta})$ and $H_1(X^\dagger)$. 
Also, let $\wt A_{\Theta}, \wt M_{\Theta}, \wt P_{\Theta}$ be the $\mathbb{Z}[H]$-valued matrices appearing in the characterization of $\mf m$ from \Cref{thm:mcpoly_det_formula} and the determinant formula for $\Delta_\Gamma$ (\Cref{th:det_form}), associated to $f_\Theta$. 

Suppose that the class $u$ is positively orientable so that by \Cref{prop:all_orient}
the map $f_\Theta$ is also positively orientable.
This implies that $\wt A_{\Theta} = \wt M_{\Theta}$ just as in \Cref{th:polynomials} and using \Cref{th:det_form} and \Cref{lem:vertex_poly_formula} we conclude that
\[
\mf m \doteq \Delta_\Gamma \cdot \mf p.
\]

Now let $w$ be any other primitive integral class in $\mc C_{\mf m}$. The above formula gives the equality of specializations: $\mf m^w(t) \doteq \Delta_\Gamma^w(t) \cdot \mf p^w(t)$. Moreover, from the definition of $\mf p$ we see that all the zeroes of $\mf p^w(t)$ are roots of unity. Hence, if $\lambda>1$ is equal to $\lambda(w)$, then $\lambda^{-1}$ is the smallest root of $\mf m^w(t)$ \eqref{eqn:specialized_det_formula} and hence of $\Delta_\Gamma^w(t)$. By \Cref{lem:spec_char_poly}, this gives that $\lambda$ is the largest root of the characteristic polynomial of any monodromy associated to $w$. 
This implies that $w$ is positively orientable and completes the proof.
\end{proof}

\subsection{Negatively orientable case}\label{sec:anti-orientable}
We now relate the polynomial invariants of a negatively orientable fully irreducible outer automorphism. Given that $\mc C_{\mf m}$ contains a negatively orientable class, we first define an involution $\iota \colon \Z[H] \to \Z[H]$. 

Let $u \in \mc C_{\mf m}$ be a negatively orientable primitive integral class. If $\Theta = \Theta_u$ is a cross section dual to $u$, then its return map $f_\Theta$ is negatively orientable by \Cref{prop:all_orient}. In this case, with notation as in the proof of \Cref{th:orientable_cone}, $\wt A_\Theta = - \wt M_\Theta$.  

Define $\epsilon_u\colon H \to \{1,-1\}$ to be the homomorphism given by setting $\epsilon_u(h) = 1$ if $u(h)$ is even and $\epsilon_u(h) = -1$ if $u(h)$ is odd. In other words, $\epsilon_u$ is the reduction of $u$ mod $2$ (where $\mathbb{Z}/2 = \{1,-1\}$) and we often write $\epsilon_u = (-1)^u$. 
Define a ring automorphism $\iota_u \colon \mathbb{Z}[H] \to \mathbb{Z}[H]$ extending  $h \mapsto \epsilon_u(h)h$. Note that $\iota_u \colon \mathbb{Z}[H] \to \mathbb{Z}[H]$ is an involution. We observe that
\begin{align} \label{eq:anti_calc}
\iota_u(\mf m) &\doteq \iota_u (\det(\dt_\Theta I - \wt A_\Theta))  \\
&= \det(-\dt_\Theta I - \wt A_\Theta)) \nonumber \\ 
&\doteq \det (\dt_\Theta I - \wt M_\Theta) \nonumber \\
& \doteq \Delta_\Gamma \cdot \mf p \nonumber 
\end{align}

A priori, however, the class $\epsilon_u$ depends on $u$. The following lemma states that each negatively orientable class in $\mc C_{\mf m}$ determines the same class, which we call the \emph{orientation class} $\epsilon \colon \Gamma \to \{-1,1\}$. We define the associated ring involution $\iota \colon \mathbb{Z}[H] \to \mathbb{Z}[H]$ accordingly.

\begin{proposition} \label{lem:or_class}
If $u$ and $\zeta$ are negatively orientable classes in $\mc C_{\mf m}$, then $\epsilon_u = \epsilon_\zeta$.
Hence, if $\mc C_{\mf m}$ contains a single negatively orientable class, then $\iota (\mf m) = \Delta_\Gamma \cdot \mf p$.
\end{proposition}

\begin{proof}
By \Cref{eq:anti_calc}, $\iota_u(\mf m) \doteq \iota_\zeta (\mf m)$. We claim that if $\mf m'$ is the normalization defined in \Cref{prop:alt_def}, then $\iota_u(\mf m') = \iota_\zeta (\mf m')$ as elements of $\Z[H]$. Indeed, we know that $\mf m' \doteq (\iota_u \circ \iota_\zeta) (\mf m')$. But both sides of this equation satisfy the normalization in \Cref{prop:alt_def} (since $\iota_u$ and $\iota_\zeta$ do not change the support) and so must be equal by uniqueness. This proves the claim.

Finally, if $\mf m' = \sum a_g g$, then equating coefficients gives that $\epsilon_u(g) = \epsilon_\zeta (g)$ for each $g \in \mathrm{supp}(\mf m')$. Since $\mathrm{supp}(\mf m')$ generates $H$ by \Cref{prop:mcgenerate}, we see that $\epsilon_u = \epsilon_\zeta$.
The second statement now follows from \Cref{eq:anti_calc}.
\end{proof}

\begin{theorem}\label{th:anti_cone}
Suppose that $\mc C_{\mf m}$ contains a negatively orientable class $u$. Then for any other primitive integral class $\zeta \in \mc C_{\mf m}$, the following are equivalent:
\begin{enumerate}
\item $\lambda(\zeta) = \rho(\zeta)$,
\item $\zeta$ is negatively orientable,
\item $\zeta = u  \text{ mod } 2$.
\end{enumerate}
\end{theorem}

Note that if $\mc C_{\mf m}$ is a cone of cross sections that contains a class $u$ such that $\lambda(u) = \rho(u)$, then $\mc C_{\mf m}$ is covered by either \Cref{th:orientable_cone} or \Cref{th:anti_cone}.

\begin{proof}
By \Cref{th:orientable_cone}, there are no positively orientable first return maps in the cone $\mc C_{\mf m}$, so $(1)$ and $(2)$ are equivalent by definition.

Now suppose that $\zeta = u  \text{ mod } 2$. Then $\epsilon = (-1)^\zeta$.
By \Cref{lem:or_class}, we know
$\iota(\mf m) \doteq \Delta_\Gamma \cdot \mf p,$
and so specializing to $\zeta$ gives
\[
\iota(\mf m)^\zeta(t) \doteq \Delta_\Gamma^\zeta(t) \cdot \mf p^\zeta(t),
\]
where $p^\zeta(t)$ is a product of cyclotomic polynomials and  $\Delta_\Gamma^\zeta(t)$ is, up to a factor of $t^k(t-1)$,
the characteristic polynomial of any monodromy associated to $\zeta$
(see \Cref{lem:spec_char_poly}). 

We claim that $\iota(\mf m)^\zeta(t) = \mf m^\zeta(-t)$. Indeed, 
if $\mf m = \sum a_g g \in \mathbb{Z}[H]$, then
\begin{align*}
\iota(\mf m)^\zeta(t) &= \sum \epsilon(g) a_g t^{\zeta(g)}\\
&= \sum (-1)^{\zeta(g)} a_g t^{\zeta(g)}\\
&= \sum a_g (-t)^{\zeta(g)}\\
&= \mf m^\zeta(-t). 
\end{align*}
Using equation \eqref{eqn:specialized_det_formula} and \Cref{lem:spec_char_poly}, we conclude as in the proof of \Cref{th:orientable_cone} that $\zeta$ is negatively orientable and hence $(3)$ implies $(2)$.
Since $(2) \implies (3)$ follows from \Cref{lem:or_class}, this completes the proof.
\end{proof}

With these facts in hand, we can prove \Cref{thm:intro-identify_orientable_classes} from the introduction.

\begin{proof}[Proof of \Cref{thm:intro-identify_orientable_classes}]
By \Cref{thm:DKL-cone-theorem}, the component $\mc C$ of $\BNS(\Gamma)$ containing the dual class of $\free*_\varphi$ is equal to the cone of cross sections $\mc C_{\mf m}$. If some primitive integral $u \in \mc C$ is positively orientable, then so is every such class in $\mc C$ (\Cref{th:orientable_cone}) and so $\lambda(u) = \rho(u)$ for the entire cone. This is case $(1)$ from the theorem statement.

If some primitive integral $u \in \mc C$ is negatively orientable, then \Cref{th:anti_cone} gives that the primitive integral $\zeta \in \mc C$ with $\lambda(\zeta) = \rho(\zeta)$ are exactly those that equal $u$ mod $2$, thus giving case $(2)$. The final alternative is that no classes in $\mc C$ are orientable and this is exactly case $(3)$.
\end{proof}

\subsection{An equation mod $2$}
\label{sec:gen_orient_polys}
Here we show that the conclusion of \Cref{th:polynomials} holds in total generality after reducing mod $2$. 

\begin{theorem} \label{th:polynomials_mod2}
Let $\varphi$ be a fully irreducible automorphism, with associated free-by-cyclic group $\Gamma$.
If $\mathrm{rank}(H_1(\Gamma)) \ge 2$, then 
\[
\mf m \doteq \Delta_\Gamma \cdot \mf p \; (\mathrm{mod}\; 2)
\]
Otherwise, the equation holds after multiplying $\mf m$ by $(\dt -1)$, where $\dt$ generates $H_1(\Gamma) / \mathrm{torsion}$.
\end{theorem}

We begin by noting that the proof of \Cref{th:det_form} holds with coefficients in $\mathbb{Z}/2$; that is, using the group ring $\mathbb{Z}/2[H]$. In short, if we denote the first Fitting invariant of $H_1(\wt X, \wt T; \mathbb{Z}/2)$ by $\Delta_2$, then 
\[
\Delta_2 \doteq (\dt-1)^l \cdot \frac{\det(\dt I -\wt M_2)}{\det(\dt I-\wt P_2)},
\]
where $l=0$ if $\mathrm{rank}(H_1(X)) \ge 2$ and $l=1$ otherwise. Here $\wt M_2$ and $\wt P_2$ are the mod $2$ reductions of $\wt M$ and $\wt P$, respectively, from \S\ref{sec:abelian_cover}. More formally, if we denote by $r \colon \mathbb{Z}[H] \to \mathbb{Z}/2[H]$ the mod $2$ reduction homomorphism and use the same notation for the corresponding ring homomorphism between matrix rings, then $\wt M_2 = r(\wt M)$ and $\wt P_2 = r(\wt P)$. 

If we denote the adjacency matrix for $\wt f$ (from \S\ref{sec:Mccone}) by $\wt A$ and similarly set $\wt A_2 = r(A)$, then $\wt A_2 = \wt M_2$ as matrices with coefficients in $\mathbb{Z}/2[H]$. Then note that
\[
r(\mf m) \doteq  r(\det(\dt I -\wt A)) = \det(\dt I -\wt A_2) =  \det(\dt I -\wt M_2),
\] 
and similarly $r(\mf p) = \det(\dt I -\wt P_2)$. We conclude that 
\[
(\dt-1)^l \cdot r(\mf m) \doteq \Delta_2 \cdot r(\mf p).
\]

It remains to prove that $r(\Delta) \doteq \Delta_2$ in $\mathbb{Z}/2[H]$, where $\Delta = \Delta_\Gamma$. First, note that by the universal coefficients theorem and the fact that $H_0(\wt X, \wt T) = 0$,
\begin{align*}
H_1(\wt X, \wt T; \mathbb{Z}/2) &= H_1(\wt X, \wt T) \otimes \mathbb{Z}/2\\
&= H_1(\wt X, \wt T) \otimes_{\mathbb{Z}[H]} \mathbb{Z}/2[H].
\end{align*}

Then by the general theory of Fitting ideals (e.g. \cite[Corollary 20.5]{eisenbud2013commutative}),
\[
\mathrm{Fit}_1(H_1(\wt X, \wt T) \otimes_{\mathbb{Z}[H]} \mathbb{Z}/2[H]) = r(\mathrm{Fit}_1(H_1(\wt X, \wt T)) ).
\]

To conclude we need the following corollary to the proof of \Cref{th:det_form}. 
It is similar to \cite[Theorem 5.1]{McMullen--Alexander_polynomial} in the $3$-manifold setting.

\begin{corollary}\label{cor:ideal}
Let $\mc A \subset \Z[H]$ be the augmentation ideal.
Then the Alexander ideal equals
\[
\mathrm{Fit}_1(H_1(\wt X, \wt T)) = (\Delta) \cdot \mc A^p,
\]
where $(\Delta)$ is the ideal generated by $\Delta$, and where
$p=1$ if $\mathrm{rank}(H_1(X))\ge 2$ and $p=0$ if $\mathrm{rank}(H_1(X)) =1$.
\end{corollary}

\begin{proof}

Using notation from \Cref{th:det_form}, we have $\mathrm{Fit}_1(H_1(\wt X, \wt T)) = \left( \Delta_1, \ldots, \Delta_{|E|+1}\right)$. Fix some index, say $k$, representing a vertical edge $\eta_k$ of $\mc E \setminus T$ so that $\wt e_k = \wt \eta_k$. Then from the proof of \Cref{th:det_form} we have 
\begin{align*}
(\dt-1)^{1-p} \Delta_k = \alpha(\wt \eta_k) \Delta_{\Gamma}.
\end{align*}
Now using \Cref{lem:eko} to relate $\Delta_i$ and $\Delta_k$, for every index $i\in \{1,\dots,\abs{E}+1\}$ we get 
\begin{align*}
(\dt-1)^{1-p} \Delta_i = \alpha(\wt e_i) \Delta_{\Gamma}.
\end{align*}
 
Since the image of $\alpha$ is the augmentation ideal, for $\mathrm{rank}(H_1(X))\ge 2$ and $1-p=0$ we get 
\[\left( \Delta_1, \ldots, \Delta_{|E|+1} \right ) = (\Delta_{\Gamma}) \cdot \mc A.\] 
For $\mathrm{rank}(H_1(X)) =1$ and $1-p = 1$,
we showed in the proof of \Cref{th:det_form} that the gcd of $\alpha(\wt e_i)$ is equal to $(\dt -1)$. Since $\Z[\dt^{\pm 1}]$ is a PID, we have $\mc A = ((z-1))$ and thus
\[\left( \Delta_1, \ldots, \Delta_{|E|+1} \right ) = (\Delta_\Gamma). \qedhere\]
\end{proof}

Hence, we see that
\[
\mathrm{Fit}_1(H_1(\wt X, \wt T; \mathbb{Z}/2)) = r(\Delta) \cdot r(\mc A^p),
\]
and so $\Delta_2 = r(\Delta) \cdot \mathrm{gcd} \; r(\mc A^p)$. 

If $p=0$, then clearly $\Delta_2 = r(\Delta)$.
Otherwise, $p=1$ and $r(\mc A)$ is the ideal in $\mathbb{Z}/2[H]$ generated by $1-g_i$, where recall the elements $g_i$ from the proof of \Cref{th:det_form}.  
In this case, however, there are $g_j$ and $g_k$ that generate a rank $2$ subgroup of $H$ and so $(1-g_j)$ and $(1-g_k)$ are still relatively prime in $\mathbb{Z}/2[H]$.
We conclude that $\mathrm{gcd} \; r(\mc A) = 1$. This shows that $\Delta_2 = r(\Delta)$ and completes the proof that
\[
(\dt-1)^l \cdot r(\mf m) \doteq r(\Delta) \cdot r(\mf p).
\]

\subsection{Newton polytopes and the cone of sections}
\label{sec:newton-polytope}
We conclude with an observation, which will be useful in \S\ref{sec:examples} below, that in the orientable case the cone of sections can be computed directly from the Alexander polynomial. Recall that the \define{Newton polytope} of an element $\mf q\in \Z[H]$ is the convex hull $N(\mf q) \subset H_1(X;\R)$ of the elements $h\in H$ appearing with nonzero coefficient in $\mf q$ (see \cite[Appendix]{McMullen--poly_invariants}). The \define{dual cone} of a vertex $v\in N(\mf q)$ is by definition the set of cohomology classes $u\in H^1(X;\R)$ that achieve a maximum value on $N(\mf q)$ precisely at the vertex $v$.

Letting $\mathrm{inv} \colon \Z[H] \to \Z[H]$ be the homomorphism sending $h$ to $h^{-1}$, \Cref{thm:DKL-cone-theorem} says that the cone of cross sections $\mc C_{\mf m}$ of the folded mapping torus $X^\dagger$ (which is also a component of $\BNS(\Gamma)$) is the dual cone of a vertex of $N(\mathrm{inv}(\mf m))$. When cone of sections contains an orientable class, the same holds for the Alexander polynomial:

\begin{lemma}
\label{lem:cone_via_newtwon_poly}
Suppose the cone $\mc C_{\mf m}$ contains an orientable primitive class, that is a class with $\lambda(u) = \rho(u)$. Then $\mc C_{\mf m}$ is equal to the dual cone of a vertex of $N(\mathrm{inv}(\Delta_{\Gamma}))$ (namely the unique vertex whose dual cone contains $u$).
 \end{lemma}
\begin{proof}
Let $\iota\colon \Z[H]\to \Z[H]$ be the involution defined in \S\ref{sec:anti-orientable} in case that $u$ is negatively orientable, and let $\iota$ denote the identity if $u$ is positively orientable.
By \Cref{th:polynomials} and \Cref{lem:or_class}, we thus have  $\mf m \doteq \iota(\Delta_{\Gamma})\cdot \iota(\mf p)$.
By \Cref{th:orientable_cone} and \Cref{th:anti_cone} we also have
 $\lambda(\zeta) = \rho(\zeta)$ for every primitive integral $\zeta \in \mc C_{\mf m}$ that agrees with $u$ mod $2$. Thus \Cref{lem:spec_char_poly} says $\lambda(\zeta)$ is the reciprocal of the smallest root of the specialization $ \iota(\Delta_{\Gamma})^\zeta(t)$ or, equivalently, the largest root of the specialization of $\mathrm{inv}(\iota(\Delta_{\Gamma}))$ at $\zeta$ (see \Cref{rmk:mcMpoly_under_inv}). By \cite{DKL2}, $\log(\lambda(\zeta))$ tends to infinity as $\zeta$ tends to the boundary of $\mc C_{\mf m}$. 
Now by using  \cite[Theorem A.1]{McMullen--poly_invariants}, we conclude that $\mc C_{\mf m}$ is equal to the dual cone of a vertex of $N(\mathrm{inv}(\iota(\Delta_{\Gamma}))) = N(\mathrm{inv}(\Delta_\Gamma))$. 
 \end{proof}

\section{Examples}\label{sec:examples}
In this section we give several fully irreducible automorphisms $\varphi$ that illustrate aspects of the theory developed in the paper.
These also serve to contrast the situation for surface homeomorphisms and to highlight that the four stretch factors $\lambda_\varphi$, $\rho_\varphi$, $\lambda_{\varphi\inv}$, $\rho_{\varphi\inv}$ of an automorphism and its inverse are in general unrelated to each other. To this end, we say $\varphi$ has \emph{stretch factor symmetry} if $\lambda_\varphi = \lambda_{\varphi\inv}$.

Recall that a fully irreducible automorphism $\varphi$ of a free group is \emph{geometric} if it may be represented by a pseudo-Anosov homeomorphism on a (punctured) surface. In this case, the pseudo-Anosov property automatically implies $\lambda_{\varphi} = \lambda_{\varphi\inv}$, and one additionally has $\rho_\varphi = \rho_{\varphi\inv}$ provided the surface itself is orientable. These basic equalities can be deduced from the stronger facts that the Teichm\"{u}ller \cite{McMullen--poly_invariants} and Alexander \cite{blanchfield1957intersection, turaev1975alexander} polynomials of the fibered $3$-manifold determined by $\varphi$ are symmetric. In addition, if $\varphi$ is orientable then $\lambda_{\varphi} = \rho_\varphi$ and $\lambda_{\varphi^{-1}} = \rho_{\varphi^{-1}}$.  

The fully irreducible automorphism $\varphi$ is known to be geometric if and only if both its attracting and repelling trees are geometric \cite{guirardel2005coeur, handel2007expansion}; see also \cite{coulbois2012botany}.
The automorphism is \emph{parageometric} when the attracting tree is geometric but the repelling tree is not, and in this case  Handel--Mosher have shown  $\lambda_{\varphi} \neq \lambda_{\varphi^{-1}}$ \cite{handel2007expansion}.

The examples below show that orientable fully irreducible automorphisms need not exhibit the stretch factor symmetry enjoyed by pseudo-Anosovs, and that orientability of $\varphi$ and $\varphi\inv$ are in general independent. Throughout, we use Thierry Coulbois's train track Sage package \cite{coulbois} to check that the example graph maps represent fully irreducible automorphisms and to compute their stretch factors. The package can also be used to find periodic Nielsen paths in a stable train track representative of a fully irreducible automorphism $\varphi$. We then use \cite[Theorem 3.2]{BF:OuterLimits} to determine if the attracting tree $T_{\varphi}^+$ is geometric or nongeometric: $T_{\varphi}^+$ is geometric if and only if a stable train track representative contains an indivisible orbit of periodic Nielsen paths.

\begin{example}\label{example:O and none}
Here we give an example of a parageometric fully irreducible automorphism that is (positively) orientable but whose inverse is not orientable.

Let $\varphi$ be the fully irreducible automorphism of $F_3 = \langle a,b,c \rangle$ given by 
\[a\stackrel{\varphi}{\to} abbc,\qquad b\stackrel{\varphi}{\to} bcabbc,\qquad c \stackrel{\varphi}{\to} CBcabbc,\]
where throughout we use capital letters to denote inverse elements or reversed edges.
This is represented by a train track graph map $\map\colon G\to G$, where $G$ is the graph with two vertices $0,1$ and four edges $0 \stackrel{a}{\to} 0$, $0 \stackrel{b}{\to}1$, $1 \stackrel{c}{\to} 0$, and $1 \stackrel{d}{\to} 0$, 
where the marking is $a \leftrightarrow a$, $b\leftrightarrow bd$, $c\leftrightarrow Dc$, 
and where the map $\map$ is defined on edges by 
\[\map(a) = abdbc,\qquad  g(b) = bcab,\qquad g(c) = cabdbc,\qquad g(d) = dbc.\]
This graph map is clearly positively orientable, and thus $\lambda_\varphi = \rho_\varphi\sim 4.61$ both equal the largest root of the characteristic polynomial $1-3t+7t^2-6t^3+t^4$.
The inverse automorphism $\varphi\inv$
is represented by a train track map $g'$ on a graph $G'$ with two vertices $0,1$ and four edges $0 \stackrel{a}{\to} 1$, $0 \stackrel{b}{\to} 1$, $1 \stackrel{c}{\to} 0$, $0 \stackrel{f}{\to} 1$. Here the marking is $a \leftrightarrow aF$, $b \leftrightarrow bF$, $c \leftrightarrow fc$ and the map $g'$ is given by 
\[g'(a) = acf, \qquad g'(b) = bFCFa, \qquad g'(c) = AfcfA, \qquad g'(f) = b.\]
From this one easily calculates that $\lambda_{\varphi\inv} \sim 3.08$ but that $\rho_{\varphi\inv} \sim 2.15$. According to \Cref{thm:orientable-auto}, neither the graph map $g'$ nor automorphism $\varphi\inv$ is orientable.
\end{example}

\begin{example}\label{example:Anti anti}
Our next example shows that even if a fully irreducible outer automorphism and its inverse are both (negatively) orientable, they need not be geometric. In fact, here $\varphi^{-1}$ is parageometric with $\lambda_{\varphi} \sim 2.17$ and $\lambda_{\varphi^{-1}} \sim 3.72$. Note that squaring provides a positively orientable pair $\varphi^2, \varphi^{-2}$ that also fails to be geometric.

For $F_3 = \langle a, b, c \rangle$, let $\varphi$ be  given by 
$a \to BA$, $b \to CAA$, $c \to B$. 
This map is in fact a negatively orientable train track map on the 3-petal rose labeled $a,b,c$. 
The automorphism $\varphi^{-1}$ is given by 
$a \to Ac$, $B \to c$, $C \to bAcAc$, 
which is also a negatively orientable train track representative. 
The characteristic polynomials of $\varphi$ and $\varphi^{-1}$ on homology are $-1+3t+t^2-t^3$ and $-1+3t+3t^2-t^3$, respectively. 
\end{example}

\begin{example}
\label{sec:AntiO}
Our final example provides fully irreducible outer automorphisms $\varphi$ and $\varphi^{-1}$ that are negatively orientable and satisfy stretch factor symmetry. That is, $\lambda_{\varphi} = \lambda_{\varphi^{-1}} = \rho_{\varphi} = \rho_{\varphi^{-1}} \sim 3.73$ and yet $\varphi$ is not a geometric automorphism. Indeed, the attracting and repelling trees associated to $\varphi$ are nongeometric and the McMullen polynomial of the free-by-cyclic group $\Gamma$ associated to $\varphi$ is not symmetric. 

We also use this example to detail the computation of the Alexander and McMullen polynomials using \Cref{th:det_form} and \Cref{lem:or_class}, and to illustrate 
how orientability varies in the cone of sections in keeping with \Cref{th:anti_cone}.

\begin{figure}[ht]
    \centering
    \def\svgwidth{.45\columnwidth}
    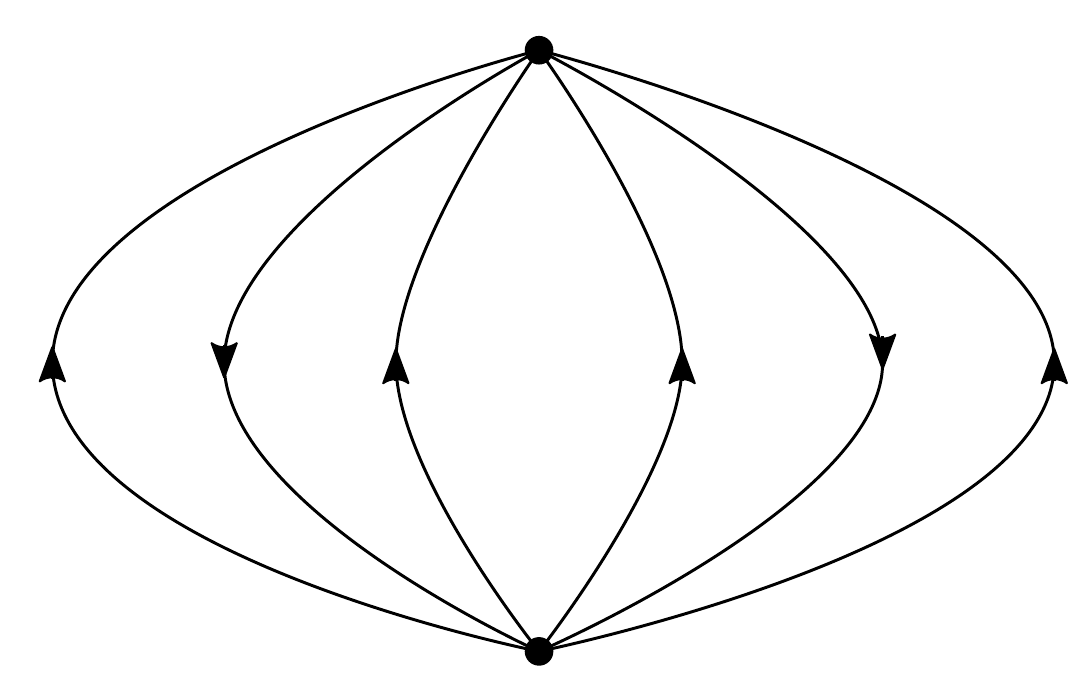
    \caption{Graph $G$}
    \label{fig:AntiO_G}
\end{figure}

Let $G$ be the graph in \Cref{fig:AntiO_G} with two vertices $v,w$, and with four oriented edges $a,c,d,f$ from $w$ to $v$ and two oriented edges $b,e$ from $v$ to $w$.
Let $\map \colon G \to G$ be the, clearly negatively orientable, graph map defined on edges as
\[a \mapsto DEABF,\;\; b \mapsto EDE,\;\; c \mapsto ABF,\;\; d \mapsto CEDBA,\;\; e \mapsto BAB,\;\; f \mapsto CED.\] 
Choosing $f$ as the maximal subtree gives a basis $aF, fb, cF, dF, fe$ of $F_5 =\pi_1(G,w)$; denoting these as $a, b, c, d, e$ for brevity,
the corresponding automorphism is 
\[a \stackrel{\varphi}{\mapsto} DEABdec,\; b \stackrel{\varphi}{\mapsto} CEDEDE,\; c \stackrel{\varphi}{\mapsto} ABdec,\; d \stackrel{\varphi}{\mapsto} CEDBAdec,\; e \stackrel{\varphi}{\mapsto} CEDBAB\]
 
Let $X$ be the mapping torus of $g$ and $H$ its homology. Recall from \S\ref{sec:abelian_cover} that $K$ is the image of $\pi_1(G)\to \pi_1(X)\to H$. 
A direct computation using the standard group presentation of $\Gamma = F_5 \rtimes_{\varphi} \mathbb{Z}$ gives
$[aF] = [bf]= [cF] \neq 0$ and $[dF] = [ef]=0$ 
in $H$ and hence in $K$. 
Denoting the non-trivial homology class by $\alpha$, the free abelian cover $\wt{G}_0\to G$ with deck group $K$ is therefore as depicted in \Cref{fig:AntiOGab1}. 
Fix base lifts $\tilde v, \tilde w, \tilde a, \tilde b, \dots$ of the cells of $G$ as indicated in the figure, and let  $\tilde{\map} \colon \wt{G}_0 \to \wt{G}_0$ to be the unique lift of $\map$ such that $\tilde\map (\tilde v) = \tilde w$. This determines a splitting $H = \langle \alpha \rangle \oplus \langle \dt \rangle$, as described in \S\ref{sec:abelian_cover}, where $\dt$ maps to $-1$ and $\alpha$ to 0 under the map $\pi_1(X) \to \mathbb{Z}$.

\begin{figure}[ht]
    \centering{\def\svgwidth{.5\columnwidth}}
    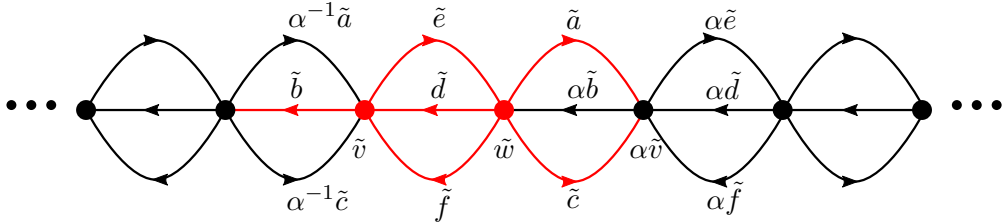
    \caption{The free abelian cover $\wt{G}_0$}
    \label{fig:AntiOGab1}
\end{figure}

We now compute the $\Z[K]$--matrices $\wt{P}$ and $\wt{M}$ describing the action of $\tilde\map$ on $C_1(\wt G_0)$. We already know $\tilde{\map}(\tilde v) = \tilde w$.
To find $\tilde{\map}(\tilde w)$ we look at the image of the edge $\tilde{e}$ with end points $\tilde{v}$ and $\tilde{w}$. Since $\wt{\map}(\tilde{v}) = \tilde{w}$, $\wt{\map}(\tilde{e})$ starts at $\tilde{w}$, traverses the edges $\alpha \tilde{B}, \tilde{A}, \alpha \tilde{B}$ in that order and ends at the vertex $\alpha \tilde{v}$. Therefore, $\wt{\map}(\tilde{w}) = \alpha \tilde{v}$. 
The images $\tilde{\map}(\tilde{a}),\tilde{\map}(\tilde{b}),\dots$ of the other base edges are found in the same manner. We thus compute:
\[\wt{P} = \begin{pmatrix}
0 & \alpha \\ 1 & 0 \end{pmatrix},
\qquad
\text{and}
\qquad
\wt{M} = \begin{pmatrix}
            -1 & 0 &-1 &-1 &-1 &0 \\
            -\alpha &0 &-\alpha &-\alpha &-2\alpha &0 \\
            0 &0 &0 &-1 &0 &-1 \\
            -\alpha &-1 &0 &-1 &0 &-1 \\
            -\alpha &-2 &0 &-1 &0 &-1 \\
            -\alpha &0 &-\alpha &0 &0 &0 
           \end{pmatrix}.
\]

From this we can easily compute the polynomial invariants associated to $\varphi$. By the determinant formula in \Cref{th:det_form}, the Alexander polynomial is
\begin{flalign*}
 \Delta_{\Gamma}  \doteq \frac{\det(\dt I-\wt{M})}{\det(\dt I-\wt{P})} 
 &  = \frac{\dt^6+2\dt^5+\dt^4(1-8\alpha)+\dt^2(8\alpha^2-\alpha)- 2\alpha^ 2\dt - \alpha^3}{\dt^2-\alpha} \\
 & = \dt^4+2\dt^3+\dt^2(1-7\alpha)+2\alpha \dt + \alpha^2. 
\end{flalign*}
The McMullen polynomial satisfies $i(\mf m) \doteq \Delta_{\Gamma} \cdot \mf p$, where 
$\epsilon \colon H= \langle \alpha \rangle \oplus \langle \dt \rangle \to \Z / 2\Z$ is the map $\epsilon(\alpha) = 1$ and $\epsilon(\dt) = -1$
and 
$i \colon \Z[H] \to \Z[H]$ is the involution given by $i(h) = \epsilon(h)h$.
Thus by \Cref{lem:or_class} we compute that 
\[
\mf m \doteq i(\Delta_{\Gamma} \cdot \mf p) \doteq \dt^6- 2\dt^5+\dt^4(1-8\alpha)+\dt^2(8\alpha^2-\alpha) +2\alpha^2 \dt - \alpha^3.
\]

One may check that the inverse $\varphi\inv$ is represented by the train track map 
\[
\begin{array}{lllll}
 a\mapsto DIFEHAB, & b\mapsto EDIFE, & c\mapsto AB, &d\mapsto GBIA,& e\mapsto BIABI, \\ 
 & f\mapsto CED, &g\mapsto H, &h\mapsto IFG,& i\mapsto IF
\end{array}
\]
on the graph with 4 vertices and 9 edges $4 \stackrel{a}{\to} 1$, $2 \stackrel{b}{\to} 4$, $0 \stackrel{c}{\to} 1$, $2 \stackrel{d}{\to} 3$, $3 \stackrel{e}{\to} 0$, $0 \stackrel{f}{\to} 1$, $4 \stackrel{g}{\to} 0$, $1 \stackrel{h}{\to} 3$, $1 \stackrel{i}{\to} 2$. This map is clearly negatively orientable and may be used to confirm the stretch factor symmetry $\lambda_{\varphi\inv}= \lambda_\varphi$.

In cohomology $H^1(\Gamma) = \langle \alpha^* \rangle \oplus \langle \dt^*\rangle$,  the splittings $F_5\ast_{\varphi}\cong \Gamma \cong F_5\ast_{\varphi\inv}$ correspond to the classes $u_0 = (0,-1)$ and $-u_0 = (0,1)$. According to \Cref{lem:cone_via_newtwon_poly} the two components of $\BNS(\Gamma)$ containing $\pm u_0$ are both dual cones of vertices of the Newton polytope of $\mathrm{inv}(\Delta_\Gamma)$, which is the convex hull of the homology classes $(0,-4)$, $(0,-3)$, $(0,-2)$, $(-1,-2)$, $(-1,-1)$, $(-2,0)$. These cones $\mc C_\pm$ for $\varphi^\pm$ are thus as depicted in \Cref{fig:tikz_cones}.
According to \Cref{th:anti_cone} every primitive class $u$ in $\mc C_+$ or $\mc C_-$ is either negatively orientable or non-orientable depending on whether $u$ agrees with $u_0$ mod $2$. Notice that while the pair $\pm u_0$ has stretch factor symmetry, this need not hold for other classes in the cones. For example, the classes $u' = (2,-3)$ and $-u' = (-2,3)$ are negatively orientable so, by \Cref{lem:spec_char_poly}, their stretch factors are the reciprocals of the smallest roots of their respective specializations of $\Delta_\Gamma$. Hence one calculates that $\lambda(u') = \rho(u') \sim 1.43092$ but that $\lambda(-u') = \rho(-u')  \sim 1.36225 $.

\begin{figure}
\begin{tikzpicture}[scale=.6]
\pgfmathsetmacro\x{6};
\pgfmathsetmacro\y{4};
\pgfmathsetmacro\xm{\x-1};
\pgfmathsetmacro\ym{\y-1};

\filldraw[opacity=.2,color=cyan] (\x,0) -- (\x,-\y) -- (-\x,-\y) -- (-\x,-.5*\x) --(0,0);
\filldraw[opacity=.2,color=lime] (0,0) -- (2*\y/3,\y) -- (-\x,\y) -- (-\x,-.5*\x);
\foreach \i in {-\xm, ..., \xm}
{\foreach \j in {-\ym, ..., \ym}
{
\draw[opacity=.1,gray,thin] (-\x,\j) -- (\x,\j);
\draw[opacity=.1,gray,thin] (\i,-\y) -- (\i,\y);
}}
\draw[gray!40!black, thick,->] (-\x,0) -- (\x+.2,0) node [above left] {$\alpha^*$};
\draw[gray!40!black, thick,->] (0,-\y) -- (0,\y+.2) node [below right] {$\dt^*$};
\draw[black,ultra thick,dashed] (0,0) --  (2*\y/3,\y) node[above] {$3\dt^*=2\alpha^*$};
\draw[black,ultra thick,dashed] (0,0) --  (\x,0);
\draw[black,ultra thick,dashed] (0,0) --  (-\x, -.5*\x) node[left] {$2\dt^*=\alpha^*$};
\draw [green!40!black] plot [only marks, mark size=3.5, mark=*] 
 coordinates {
(-4,3) (-2,3)
(-4,1) (-2,1)  
(-4,-1) (2,-1) (4,-1) 
(-4,-3) (-2,-3) (2,-3) (4,-3) 
(8,1)
};
\draw [red!70!black] plot [only marks, mark size=3, mark=square*] 
 coordinates {
(-5,3) (-1,3) (1,3)
(-5,2) (-3,2) (-1,2) (1,2)
(-5,1) (-3,1) (-1,1) 
(-1,0)
(-5,-1) (-3,-1) (-1,-1) (1,-1) (3,-1) (5,-1) 
(-5,-2) (-3,-2) (-1,-2) (1,-2) (3,-2) (5,-2)
(-5,-3) (-1,-3) (1,-3) (5,-3)
(8,-1)
};

\path (\x+.25, -\y+1.5) node[right] {$\mc C_+$};
\path (-\x-.25,\y-1.5) node[left] {$\mc C_-$};
\path (8.5,1) node[right, align=left] {negatively orientable\\ primitive classes};
\path (8.5,-1) node[right, align=left] {non-orientable\\ primitive classes};
\filldraw[green!40!black] (0,-1) circle (5pt) node[below right] {$u_0$};
\filldraw[green!40!black] (0,1) circle (5pt) node[above left] {$-u_0$};
\filldraw[green!40!black] (2,-3) circle (5pt) node[below right] {$u'$};
\filldraw[green!40!black] (-2,3) circle (5pt) node[above] {$-u'$};
\end{tikzpicture}
\caption{Cones of cross sections $\mc C_+$ and $\mc C_-$ for $\varphi$ and $\varphi\inv$. 
}
\label{fig:tikz_cones}
\end{figure}
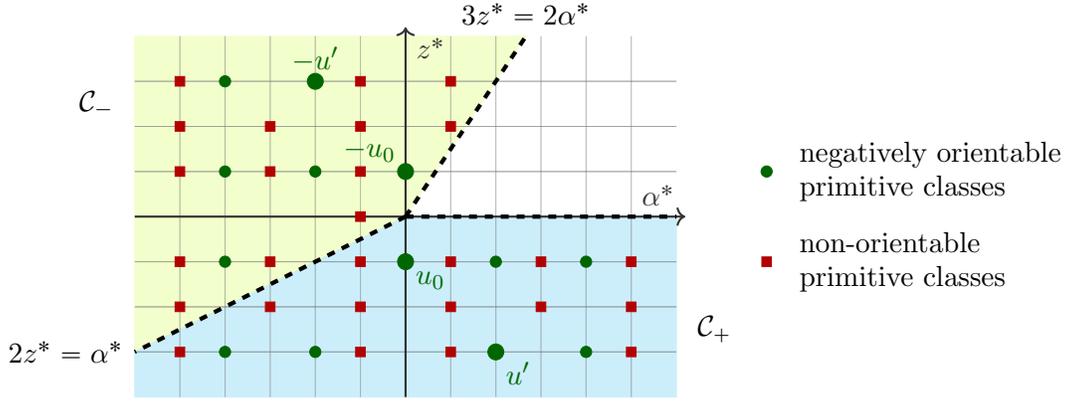
\end{example}

\bibliographystyle{alpha}
\bibliography{orientable_maps}

\end{document}